\documentclass[12pt,a4paper]{amsart}
\usepackage{amsmath}
\usepackage{paralist}
\usepackage{graphics}
\usepackage{amsfonts,amssymb}
\usepackage{color}
\usepackage{amsthm}
\usepackage{mathrsfs}
\usepackage{cite}
\usepackage{amsmath, amsthm, amssymb, mathrsfs}
\usepackage{newtxtext, newtxmath}
\allowdisplaybreaks[4]
\calclayout

\usepackage[colorlinks=true, linkcolor=blue, citecolor=red, urlcolor=magenta]{hyperref}
\date{}

\newtheorem{Theorem}{Theorem}[section]
\newtheorem{Lemma}{Lemma}[section]
\theoremstyle{remark}
\newtheorem{Remark}{Remark}[section]

\numberwithin{equation}{section}

\usepackage[titletoc]{appendix} 

\newcommand{\appendixnumbering}{%
  \renewcommand{\thesection}{Appendix \Alph{section}}%
  \renewcommand{\theTheorem}{\Alph{section}.\arabic{Theorem}}%
  \renewcommand{\theLemma}{\Alph{section}.\arabic{Lemma}}%
  \renewcommand{\theCorollary}{\Alph{section}.\arabic{Corollary}}%
  \renewcommand{\theDefinition}{\Alph{section}.\arabic{Definition}}%
  \renewcommand{\theProposition}{\Alph{section}.\arabic{Proposition}}%
  \renewcommand{\theequation}{\thesection.\arabic{equation}}%
  \setcounter{equation}{0}
}
\begin{document}

\title{ $L^1$ Behavior of the Stokes System and the Navier--Stokes System in the Half Space}

\author{Rao Li}
\address{School of Mathematical Sciences, Shanghai Jiao Tong University, Shanghai 200240, China.}
\email{lirao0909@sjtu.edu.cn}

\author{Haitao Wang}
\address{School of Mathematical Sciences, Institute of Natural Sciences, LSC-MOE, CMA-Shanghai, Shanghai Jiao Tong University, Shanghai 200240, China.}
\email{haitallica@sjtu.edu.cn}

\author{Chunjing Xie} 
\address{School of Mathematical Sciences,  LSC-MOE, CMA-Shanghai, Shanghai Jiao Tong University, Shanghai 200240, China.}
\email{cjxie@sjtu.edu.cn}

\begin{abstract}
{In this paper,  the  detailed behavior of solutions of both the Stokes and  the Navier--Stokes system in the half space is investigated when the initial data belongs to $L^1$. We first give a detailed asymptotic expansion of the solution to the Stokes system supplemented with $L^1$ initial data. With the aid of this expansion, we provide two sufficient conditions on the initial data so that the associated solutions of the Navier--Stokes system do not belong to  $L^1$, corresponding to long time and short time behaviors, respectively. Moreover, for the $n$- dimensional Navier--Stokes system, a pointwise spatial lower bound of order $n$ is derived in a conic neighborhood of the $x_n$-axis, provided that the initial data decays faster than order $n$. One of the main difficulties to get precise form of leading order term is the non-commutativity of the Leray projection operator with the Laplacian. Our strategy is to employ the special structure of the Green tensor observed by Han \cite{2018H} and a cancellation property of the initial data. This also yields the sufficient conditions on the initial data to get precise asymptotic behavior via comparing the linear and nonlinear parts. }
\end{abstract}



\maketitle


\section{Introduction and main results}\label{sec1}
 We consider  the Navier--Stokes system  in  the half space $\mathbb{R}^n_+=\{x \in \mathbb{R}^n:x= (x',x_n), x_n > 0\}$ 
 with $x'=(x_1,x_2,\dots,x_{n-1})\in \mathbb{R}^{n-1}$ and $n\geq 2$ supplemented with homogeneous Dirichlet boundary condition
\begin{equation}\label{1.1}
\left\{
\begin{aligned}
    &\partial_{t} \boldsymbol{u} - \Delta \boldsymbol{u} + \boldsymbol{u} \cdot \nabla \boldsymbol{u} + \nabla p = 0 \quad \text{in } \mathbb{R}^{n}_{+} \times (0,\infty), \\
    &\nabla \cdot \boldsymbol{u} = 0 \quad \text{in } \mathbb{R}^{n}_{+} \times (0,\infty), \\
    &\boldsymbol{u}(x,t) = 0 \quad \text{on } \partial \mathbb{R}^{n}_{+} \times (0,\infty), \\
    &\boldsymbol{u}(x,t) \to 0 \quad \text{as } |x| \to \infty, \\
    &\boldsymbol{u}(x,0) = \boldsymbol{a}(x) \quad \text{in } \mathbb{R}^{n}_{+},
\end{aligned}
\right.
\end{equation}
 where the unknown functions $\boldsymbol{u} = (u_1,u_2,\dots,u_n)$ and $p$ denote the velocity and the pressure, respectively;  $\boldsymbol{a}=(a_1,a_2,\dots,a_n)$ is a given initial data assumed to satisfy a compatibility condition: $\nabla\cdot \boldsymbol{a} = 0$ in $\mathbb{R}^n_+$ and the normal component of $\boldsymbol{a}$ equals zero on $\partial\mathbb{R}^n_+$.

In order to study (\ref{1.1}), one of the key ingredients is to study the following Stokes system in the half space supplemented with homogeneous Dirichlet boundary condition\begin{equation}\label{1.2}
\left\{
\begin{aligned}
    &\partial_{t} \boldsymbol{v} - \Delta \boldsymbol{v} + \nabla \pi = 0 \quad \text{in } \mathbb{R}^n_+ \times (0,\infty), \\
    &\nabla \cdot \boldsymbol{v} = 0 \quad \text{in } \mathbb{R}^n_+ \times (0,\infty), \\
    &\boldsymbol{v}(x,t) = 0 \quad \text{on } \partial \mathbb{R}^n_+ \times (0,\infty), \\
    &\boldsymbol{v}(x,0) = \boldsymbol{a}(x) \quad \text{in } \mathbb{R}^n_+.
\end{aligned}
\right.
\end{equation}

   In \cite{1987U}, Ukai gave the solution formula for problem (\ref{1.2}) in terms of Riesz potentials and established $L^r(\mathbb{R}^n_+)$ estimates of $\boldsymbol{v}$ and its gradient as long as the initial data $\boldsymbol{a}$ belongs to $L^q(\mathbb{R}^n_+)$ for $1<q<r<\infty$. Fujigaki and Miyakawa \cite{2001FM} improved the estimates of high order gradients of $\boldsymbol{v}$
 provided either \(1 \leq q < r \leq \infty\), or \(1 < q \leq r < \infty\). The decay properties of solutions and their gradient of the Stokes system (\ref{1.2}) in $L^1(\mathbb{R}^n_+)$, Hardy space and $L^\infty(\mathbb{R}^n_+)$ was studied in \cite{1999GSS,2006B,2008B}. The pointwise decay properties of solutions in different cases were considered in \cite{2017CJS}. For weighted $L^r$ estimates of solutions, we refer to \cite{2002B,2010J1,2010J2,2011Hw,2013KK,2014Hw,2018Hw,2024H}.
 Recently, the unrestricted Green tensor of (\ref{1.2}) and its applications were established in \cite{2023KLLT,2025KLLT}, which provided a variety of new decay results and improvements upon many existing results.

Desch, Hieber, and Prüss \cite{2001DHP} studied the corresponding resolvent problem to demonstrate the existence of a divergence-free vector field $\boldsymbol{a} \in L^1(\mathbb{R}^n_+)$ such that $\boldsymbol{v}(\cdot,t)\notin L^1(\mathbb{R}^n_+)$. Han \cite{2024H} also gave a counterexample where $\boldsymbol{v}$ does not belong to $ L^1(\mathbb{R}^n_+)$ even if $\boldsymbol{a}\in L^1(\mathbb{R}^n_+)$. A natural question is whether we can give a detailed characterization of $\boldsymbol{v}$ so that the non-integrable part of $\boldsymbol{v}$ can be addressed clearly. Our first main goal is to provide an explicit description of the solution to (\ref{1.2}) in $L^1(\mathbb{R}^n_+)$   when the initial data belongs to $L^1(\mathbb{R}^n_+)$ and satisfies some weighted estimates. 

 Before we state our main result, we first introduce the following notations.

Let $d = 1,2,3,\dots,n$ and $y = (y_1, y_2, \dots, y_d)$. Denote the heat kernel by
\[
G_t^{(d)}(y) = (4\pi t)^{-\frac{d}{2}} e^{-\frac{|y|^2}{4t}},
\]
and for $x \in \mathbb{R}^n$,
\[
G_t(x) = G_t^{(n)}(x), \qquad
E(x) =
\begin{cases}
\dfrac{1}{n(n-2)\omega_n |x|^{n-2}}, & n \geq 3,\\[6pt]
-\dfrac{1}{2\pi} \log |x|, & n = 2,
\end{cases}
\]
where $\omega_n$ denotes the volume of the unit ball in $\mathbb{R}^n$ and $E(x)$ is the fundamental
solution of $-\Delta$.

Define the matrix function $\mathscr{L}(x,t)=(L_{ij}(x,t))_{i,j=1,\dots,n}$\begin{equation}\label{L_{ij}}
    L_{ij}(x,t):=4(1-\delta_{nj})\partial_{x_j}\int_0^\infty \partial_{x_i}[G_{t+\tau}^{(n-1)}(x')G_\tau^{(1)}(x_n)]\,d\tau,
\end{equation}
where $\delta_{nj}$ is the Kronecker delta, i.e., $\delta_{nj}=1$ if $n=j$ and $\delta_{nj}=0$ otherwise.

Define\begin{equation}\label{1.6}M^*_i(x,y,t):=-4\int_{0}^{x_n}\int_{\mathbb{R}^{n-1}}\frac{\partial E(x-z)}{\partial x_{i}}G_{t}(z-y^{*})\,dz\end{equation}
and \begin{equation}\label{domain}
B(0,r) := \left\{ x \in \mathbb{R}^n \mid |x| < r \right\}, \qquad
\Lambda_\sigma := \left\{ x \in \mathbb{R}^n_+ \mid x_n > \sigma |x'|,\ \sigma > 0 \right\}.
\end{equation}

\begin{Theorem}\label{thm:1.1}
    Assume $\boldsymbol{a}\in L^1(\mathbb{R}^n_+)$, $a_n|_{\partial \mathbb{R}^n_+}=0$, $\nabla\cdot \boldsymbol{a}=0$ in $\mathbb{R}^n_+\ (n\geq 2)$ and
    $$\int_{\mathbb{R}^n_+}y_n |\boldsymbol{a}(y)|\,dy+\int_{\mathbb{R}^n_+}|y'|^2|a_n(y)|\,dy<+\infty,$$ then  the following statements hold.
    \begin{enumerate}
    \item For $i=1,2,...,n$ and $t>0$, the solution $\boldsymbol{v}(x,t)$ of  (\ref{1.2}) satisfies\begin{equation}\label{linear}
        \left\| v_i(x,t)+\sum_{j=1}^{n-1}L_{ij}(x,t)\int_{\mathbb{R}^n_+}G_t^{(1)}(y_n)y_ja_n(y)\,dy\right\|_{L^1_x(\mathbb{R}^n_+)}\leq Ct^{-\frac{1}{2}}(1+t^{-\frac{1}{2}})
    \end{equation}
 and \begin{equation}\label{high order}\begin{aligned}
        &\lim_{t\to+\infty}t^{\frac{1}{2}}\Biggl\Vert v_i+\sum_{j=1}^{n-1}L_{ij}(x,t)\int_{\mathbb{R}^n_+}G_t^{(1)}(y_n)y_ja_n(y)\,dy\\&+2\partial_{x_n}G_t(x)\int_{\mathbb{R}^n_+}y_na_i(y)\,dy+\sum_{j=1}^{n-1}\partial_{x_n}\partial_{x_j}M_i^*(x,0,t)\int_{\mathbb{R}^n_+}y_ja_n(y)\,dy\Biggr\Vert_{L^1_x(\mathbb{R}^n_+)}=0 .
    \end{aligned}\end{equation}
     \item For $i=1,2,\dots,n,j=1,2,\dots,n-1$,  $t>0$ and $x\in \Lambda_\sigma\cap \left(B(0,\sqrt{t})\right)^c$ with $\sigma$ large enough, there exist two constants $C, \widetilde{C} > 0$, independent of $t$, such that \begin{equation}\label{pointwise}\begin{aligned}
C\frac{|x_ix_j|}{|x|^{n+2}} &\leq |L_{ij}(x,t)| \leq \widetilde{C}\frac{|x_ix_j|}{|x|^{n+2}}, && \text{for } i \neq j,  \\
C\frac{1}{|x|^n} &\leq |L_{ij}(x,t)| \leq \widetilde{C}\frac{1}{|x|^n}, && \text{for } i = j. 
\end{aligned}\end{equation}
    Moreover, we have $\boldsymbol{v}(\cdot,t)\in L^1(\mathbb{R}^n_+)$ if and only if \begin{equation}\label{condition}\int_{\mathbb{R}^n_+}G_t^{(1)}(y_n)y_ja_n(y)\,dy=0\quad \text{for each }j=1,2,...,n-1.\end{equation}\end{enumerate}

\end{Theorem}
\begin{Remark}
     The estimate (\ref{linear}) and part (2) in Theorem \ref{thm:1.1} provide a description of the non-integrable part of the solution.
\end{Remark}

\begin{Remark}
   (\ref{high order}) gives an expansion of the next order terms in $L^1(\mathbb{R}^n_+)$, which is formally similar to the expansion in \cite{2001FM} for $L^r(\mathbb{R}^n_+)$ with $1<r< \infty$, while their expansion concerns the first-order terms.
\end{Remark}

It was first posed by Leray in \cite{1934L}  whether the  weak solutions of the Cauchy problem  for the Navier--Stokes system decay to zero in $L^2$ as $t$ tends to infinity when the initial data belongs to $L^2$. Extensive studies addressed decay properties in last several decades. For the solutions in the whole space $\mathbb{R}^n$, Schonbek made a series of contributions to the energy decay of solutions \cite{1985S,1991S,1992S}. The pointwise estimate of the strong solution 
\[|\boldsymbol{u}(x,t)|\leq \frac{C}{(1+|x|)^\alpha(1+t)^\frac{\beta}{2}},\quad  \ 0<\alpha+\beta\leq n+1
\]
was established by Miyakawa in \cite{2002M}.  Brandolese \cite{2004B} investigated the enhancement of decay rates induced by symmetry properties of initial data. Brandolese and Vigneron \cite{2007BV} gave a new asymptotic expansion which shows the algebraic decay of order $n+1$ when the initial data has mild decay. They have provided two sufficient conditions for $|\boldsymbol{u}|>\frac{C(t)}{|x|^{n+1}}$ in a short time. For the solutions in the half space $\mathbb{R}^n_+$, Fujigaki and Miyakawa \cite{2001FM} applied Ukai's formula to obtain the solution profile of the Stokes and Navier--Stokes systems in $L^r(\mathbb{R}^n_+)$ for $1<r<\infty$. If $n\geq 3,$ Crispo and Maremonti \cite{2006CM} showed the local existence of solution satisfying \[
|\boldsymbol{u}(x,t)| \leq \frac{C}{(1+|x|)^\alpha (1+t)^{\beta/2}},
\quad \alpha + \beta = \mu\ \text{for } \mu \in (\tfrac12,\, n).\]when $(1+|x|)^{\mu}\boldsymbol{a}\in L^\infty(\mathbb{R}^n_+)$ and for $\mu\in [1,n)$ the global  existence holds for $\mu\in [1,n)$ as long as $(1+|x|)^{\mu}\boldsymbol{a}$ is small enough in $L^\infty(\mathbb{R}^n_+)$. The local existence was improved in \cite{2017CJ} for $n\geq 2$ and $\mu\in(0,n]$. Utilizing the unrestricted Green tensor, the behavior of solutions has also been considered in \cite{2023KLLT,2025KLLT} in various function spaces, including uniformly local $L^r$, space with mixed-type pointwise decay, and space with pointwise decay alongside boundary vanishing. For $1 \leq r \leq \infty$, the $L^r$ estimates of higher-order derivatives of solutions, as well as the weighted $L^r$ estimates of solutions and their derivatives, can be found in \cite{2009HW,2012H,2014Hw,2016H,2018Hw,2024H}. Furthermore, Han \cite{2016H} also established the decay estimate of solutions in $L^{1}(\mathbb{R}^{n-1}\times(0,\sqrt{t}))$.

Similar to the linear case, $L^1(\mathbb{R}^n_+)$ also serves as a critical space for the solution of the nonlinear problem (\ref{1.2}). Subsequently, Han \cite{2014H} characterized the behavior of solutions to the Navier--Stokes system in $L^1(\mathbb{R}^n_+)$ under the \emph{tangential parity condition}. Han \cite{2018H} further elucidated the $L^1$ properties of the nonlinear term by extracting coefficients that depend on $\boldsymbol{u}$, each multiplied by a characteristic non-$L^1$ factor. In addition to solution of the Stokes system (\ref{1.2}) and this expansion, the remainder terms were shown to exhibit a decay rate of order $t^{-\frac{\alpha}{2}}$, where $\alpha$ could be any number which belongs to $(0,1)$. More recently, Han \cite{2022H} exploited the net force on the boundary to derive a full expansion of the non-$L^1$ part of the solution, which carries a clear physical interpretation. If the initial data just belongs to $L^1$, all these studies demonstrate that in many cases, the solution may fail to belong to $L^1(\mathbb{R}^n_+)$ unless the initial data satisfies a specific symmetry condition. A natural question therefore arises: can one identify suitable conditions on the initial data that characterize the non-$L^1$ behavior of the solution? 
Our second main result, Theorem \ref{thm:1.2}, addresses this question by providing explicit initial conditions under which the solution $\boldsymbol{u}$ of problem (\ref{1.1}) fails to belong to $L^1(\mathbb{R}^n_+)$. Moreover, we establish a lower bound on its algebraic decay of order $n$, thereby demonstrating that the corresponding upper bound $|\boldsymbol{u}| \leq C(1+|x|)^{-n}$ obtained in \cite{2017CJ} is in fact optimal.

For given $\boldsymbol{a}(x)$ and $\boldsymbol{u}(x,t)$, define the vector function $\mathscr{A}[\boldsymbol{a},\boldsymbol{u}](t)=\left(\mathscr{A}_j[\boldsymbol{a},\boldsymbol{u}](t)\right)_{j=1,2,\dots,n}$
\begin{equation}\label{1.8}
\mathscr{A}_j[\boldsymbol{a},\boldsymbol{u}](t):=-\int_{\mathbb{R}^n_+}G_t^{(1)}(y_n)y_j a_n(y)\,dy
+\int_0^t\int_{\mathbb{R}^n_+}G_{t-s}^{(1)}(y_n)u_n u_j(y,s)\,dy\,ds
\end{equation}and
\begin{equation}\label{qingk}
    K(\boldsymbol{a}) :=
    \begin{cases}
        \|\boldsymbol{a}\|^2_{L^1(\mathbb{R}^n_+)} + \|\boldsymbol{a}\|^4_{L^1(\mathbb{R}^n_+)} + \|\boldsymbol{a}\|^2_{L^2(\mathbb{R}^n_+)} + \|\boldsymbol{a}\|^4_{L^2(\mathbb{R}^n_+)}, & n \geq 3, \\[2.5ex]
        \begin{aligned}
            &\|\boldsymbol{a}\|^2_{L^1(\mathbb{R}^n_+)} + \|\boldsymbol{a}\|^4_{L^1(\mathbb{R}^n_+)} + \|\boldsymbol{a}\|^2_{L^2(\mathbb{R}^n_+)} + \|\boldsymbol{a}\|^4_{L^2(\mathbb{R}^n_+)} \\
            &\quad + \left(\int_{\mathbb{R}^n_+} y_n |\boldsymbol{a}(y)| \, dy\right)^2 + \left( \int_{\mathbb{R}^n_+} y_n |\boldsymbol{a}(y)| \, dy \right)^4,
        \end{aligned}
        & n = 2.
    \end{cases}
\end{equation}

\begin{Theorem}\label{thm:1.2} Assume the initial data $\boldsymbol{a} \in L^1(\mathbb{R}^n_+)\cap L^n(\mathbb{R}^n_+)\ (n\geq 2)$ satisfies
\begin{equation}\label{origin}a_n|_{\partial\mathbb{R}^n_+} = 0,\ \nabla \cdot \boldsymbol{a} = 0,\quad \text{and}\quad
x_n \boldsymbol{a}, |x'|^2 a_n \in L^1(\mathbb{R}^n_+).\end{equation}
Then there exists a $T>0$ and a unique strong  solution $\boldsymbol{u}\in C([0,T);L^n_{\sigma}(\mathbb{R}^n_+))$ of (\ref{1.1}). Denote
\[
\mathscr{R}(x,t)=\boldsymbol{u}(x,t)-\mathscr{L}(x,t)\mathscr{A}[\boldsymbol{a},\boldsymbol{u}](t),
\]
where  $\mathscr{L}$ is defined in (\ref{L_{ij}}) and $\left(\mathscr{L}(x,t)\mathscr{A}[\boldsymbol{a},\boldsymbol{u}](t)\right)_i = \sum_{j=1}^{n} L_{ij}(x,t) \mathscr{A}_j[\boldsymbol{a},\boldsymbol{u}](t)$ for $i = 1, 2, \dots, n$.

If $\mathscr{R}(\cdot,t)\in L^1(\mathbb{R}^n_+)$, then $\boldsymbol{u}(\cdot,t) \in L^1(\mathbb{R}^n_+)$ if and only if $\mathscr{A}_j[\boldsymbol{a},\boldsymbol{u}](t) = 0$ for each $j = 1,2, \dots, n-1$.
Furthermore, we have $\mathscr{R}(\cdot,t)\in L^1(\mathbb{R}^n_+)$ in  following three cases.

\begin{enumerate}
    \item  \textbf{(Long time behavior)} If $\boldsymbol{a}$ satisfies\[(1 + |x|)\boldsymbol{a} \in L^1(\mathbb{R}^n_+)\quad \text{and}\quad|x|\boldsymbol{a}, (1 + |x|)\nabla \boldsymbol{a} \in L^2(\mathbb{R}^n_+),\]
and additionally $\|\boldsymbol{a}\|_{L^n(\mathbb{R}^n_+)}$ is sufficiently small when $n \geq 3$,
then there exists a global solution $\boldsymbol{u}$ of (\ref{1.1}) and $\mathscr{R}(x,t)$ satisfies for
 $t > 0$ 

\begin{equation}\label{1.10}
\|\mathscr{R}(\cdot,t)\|_{L^1(\mathbb{R}^n_+)} \leq Ct^{-\frac{1}{2}}(1 + t^{-\frac{1}{2}}).
\end{equation}

Moreover, there exists an $\epsilon > 0$ if for some $j=  1,2, \dots, n-1$, $\int_{\mathbb{R}^n_+} y_j a_n(y)\,dy \neq 0$ and

\begin{equation}\label{epsilon}
\frac{K(\boldsymbol{a})} {\left|\int_{\mathbb{R}^n_+} y_j a_n(y)\,dy\right|} < \epsilon,
\end{equation}
then there exists a $T_1 > 0$ such that  $\boldsymbol{u}(\cdot,t) \notin L^1(\mathbb{R}^n_+)$ for $t > T_1$.
    
    \item  \textbf{(Pointwise behavior)} If $\boldsymbol{a}$ satisfies \begin{equation}\label{cond:decay}
|\boldsymbol{a}(x)| \leq \frac{C}{(1+|x|)^{\theta}} \quad \text{for some } \theta > n,
\end{equation}
then there exists a $T_2\in (0,T]$ such that (\ref{1.10}) holds for $t\in(0,T_2)$. 
Moreover,  for  $t\in (0,T_2)$, there exist $C(t)>0$ and $M(t) > 1$ such that for $x\in \Lambda_{M(t)}\cap \left(B(0,M(t))\right)^c$
\begin{equation}\label{lower}|\boldsymbol{u}(x,t)|\geq\frac{C(t)}{|x|^n}\end{equation}if $\mathscr{A}_j[\boldsymbol{a},\boldsymbol{u}](t)\neq 0$ for some $j=1,2,\dots,n-1$.
    \item  \textbf{(Short time behavior)} If  $\boldsymbol{a}$ satisfies \begin{equation}\label{p0}|\boldsymbol{a}(x)| \leq \frac{Cx_n^{b_1}}{(1+x_n)^{b_1}(1+|x|)^{\bar{b}}}\quad \text{for } \bar{b} > n \text{ and } b_1 \in [0, 1],\end{equation}then there exists a $T_3\in (0,T]$ such that (\ref{1.10}) holds for $t\in(0,T_3)$. Moreover, if $a_n$ satisfies for some $b_2 \in [0, 2b_1+2)$, $\epsilon_0,C > 0$ and some $j=1,2,...,n-1$,\begin{equation}\label{condition1}
\int_{\mathbb{R}^{n-1}} x_j a_n(x) \, dx' \geq C x_n^{b_2} \quad \text{for all } x_n \in (0, \epsilon_0),
\end{equation}
or
\begin{equation}\label{condition2}
\int_{\mathbb{R}^{n-1}} x_j a_n(x) \, dx' \leq -C x_n^{b_2} \quad \text{for all } x_n \in (0, \epsilon_0),\end{equation}  then there exists a $T_4>0$ such that for any $t\in (0,T_4)$, $\boldsymbol{u}(\cdot,t)\notin L^1(\mathbb{R}^n_+)$.
\end{enumerate}

\end{Theorem}

\begin{Remark}
The Cases (1) and (3) in Theorem \ref{thm:1.2} provide sufficient conditions for the initial data so that $\boldsymbol{u}(x,t)$ does not belong to $L^1(\mathbb{R}^n_+)$ which correspond to the long time and short time behaviors of solutions respectively. Specifically, Case (1) indicates that the smallness of the ratio of some second-order quantities to a first-order quantity of the initial data guarantees the non-$L^1$ property over long times. Case (3) shows that if the normal component of the initial data does not vanish too rapidly near the boundary, then the solution no longer belongs to $L^1(\mathbb{R}^n_+)$ over a short interval of time starting from $t=0$, even though the initial data itself may be well localized.\end{Remark} \begin{Remark}
   The estimate (\ref{1.10}) improves upon the corresponding result in \cite[Theorem 1.4]{2018H}, where the decay rate was $t^{-\alpha/2}$ for $\alpha \in (0,1)$, to the sharper decay rate $t^{-1/2}$.
\end{Remark}\begin{Remark}In Case (2), the selection of different coordinate systems (i.e. the origin may be translated arbitrarily along the
boundary) reveals that the decay of the Navier--Stokes system in the half space maintains an algebraic decay rate of order $n$ orthogonal to the boundary, as opposed to the algebraic decay rate of order $n + 1$ observed in the whole space (see \cite{2007BV}). The coefficient $C(t)$ vanishes as $t$ tends to zero since initial data satisfies $|\boldsymbol{a}(x)|\leq\frac{C}{(1+|x|)^{\theta}},\theta>n$.\end{Remark} \begin{Remark}The existence of  initial data in Case (3) of Theorem \ref{thm:1.2} follows from Lemma \ref{A.2} in the Appendix. 
\end{Remark}

By  Cases (1) and (2) in Theorem \ref{thm:1.2}, we show that the leading terms in both $L^1$ and the spatial expansion in $\Lambda_\sigma$ are given by
$\sum_{j=1}^{n-1}L_{ij}(x,t)\mathscr{A}_j[\boldsymbol{a},\boldsymbol{u}](t)$.
More precisely, see  (\ref{pointwise}),  for $i=1,2,...,n,j=1,2,\dots,n-1,t>0,\frac{|x|}{\sqrt{t}}>1$ and $x\in\Lambda_{\sigma}$ with $\sigma>1$ large enough\begin{align}
C|\mathscr{A}_j[\boldsymbol{a},\boldsymbol{u}](t)|\frac{|x_ix_j|}{|x|^{n+2}} &\leq |L_{ij}(x,t)\mathscr{A}_j[\boldsymbol{a},\boldsymbol{u}](t)| \leq \widetilde{C}|\mathscr{A}_j[\boldsymbol{a},\boldsymbol{u}](t)|\frac{|x_ix_j|}{|x|^{n+2}}, && \text{for } i \neq j,  \\
C|\mathscr{A}_j[\boldsymbol{a},\boldsymbol{u}](t)|\frac{1}{|x|^n} &\leq |L_{ij}(x,t)\mathscr{A}_j[\boldsymbol{a},\boldsymbol{u}](t)| \leq \widetilde{C}|\mathscr{A}_j[\boldsymbol{a},\boldsymbol{u}](t)|\frac{1}{|x|^n}, && \text{for } i = j. 
\end{align}
In fact, we can also show the rate of convergence of $\mathscr{A}_j(t)$ as $t$ tends to zero. 
\begin{Theorem}[Initial vanishing of leading term]\label{thm:1.3}
If  the initial data $\boldsymbol{a}\in L^1(\mathbb{R}^n_+)\cap L^n(\mathbb{R}^n_+)$ satisfies (\ref{origin}) and
\begin{equation}\label{p1}
|a_n(x)| \leq \frac{C x_n^{b}}{(1+x_n)^{b} (1+|x|)^{\bar{b}}}
\quad \text{with } \bar{b}>n\text{ and }b\geq 0,\end{equation} then there exists a $T'\in (0,1)$ and a unique strong solution $\boldsymbol{u}$ of (\ref{1.1}) such that  
\[
|\mathscr{A}_j[\boldsymbol{a},\boldsymbol{u}](t)| \leq C t^{\min\{\frac{1}{2},\frac{b}{2}\}} \quad \text{for } t \in (0,T') \text{ and } j = 1,2,\dots,n-1.\]
 If the initial data $\boldsymbol{a}\in L^1(\mathbb{R}^n_+)\cap L^n(\mathbb{R}^n_+)$ satisfies (\ref{origin}) and \begin{equation}\label{p2}
\begin{cases}
|\boldsymbol{a}(x)| \leq \dfrac{C x_n^{b_1}}{(1+x_n)^{b_1}\,(1+|x|)^{\bar{b}}},\\[6pt]
|a_n(x)| \leq \dfrac{C x_n^{b_2}}{(1+x_n)^{b_2}\,(1+|x|)^{\bar{b}}},
\end{cases}
\quad\text{with } \bar{b}>n,\ b_1\in[0,1]\text{ and } b_2\in[0,4],
\end{equation} then there exists a $T''\in (0,1)$ and a unique strong solution $\boldsymbol{u}$ of (\ref{1.1}) such that  
\[
|\mathscr{A}_j[\boldsymbol{a},\boldsymbol{u}](t)| \leq C t^{\min\{b_1+1,\frac{b_2}{2}\}}
\quad \text{for } t \in (0,T'') \text{ and } j = 1,2,\dots,n-1.\]
    
\end{Theorem}\begin{Remark}
    For the solutions in the whole space with well localized data, and for small $t$ and large enough $|x|$, the initial vanishing of leading term can be described by the pointwise estimate $C\frac{t}{|x|^{n+1}}\leq|\boldsymbol{u}(x,t)|\leq \widetilde{C}\frac{t}{|x|^{n+1}}$  (\!\cite{2007BV}). However, the vanishing rate of leading term in the half space is related to the boundary vanishing rate of initial data.
\end{Remark}
In the following we give the key ideas for the proof of main results. In \cite{2018H}, Han observed that the commutation of normal derivatives in the Green tensor helps extract the non‑$L^1$ part (see Lemma \ref{lem:2.4}) and gave an expansion of the nonlinear part in integral form (\ref{duhamel}). To apply this property to the linear part, we make use of a natural cancellation property of the normal component of the initial data (see Lemma \ref{A.1}) and finally extract the leading terms in Theorem \ref{thm:1.1} and \ref{thm:1.2}. Then the sufficient conditions on the initial data for $\boldsymbol{u}(\cdot,t)\notin L^1(\mathbb{R}^n_+)$ in Theorem \ref{thm:1.2} are obtained by comparing the linear and nonlinear parts. The improved decay estimate (\ref{1.10}) for the remainder terms in Cases (1) and (2) of Theorem \ref{thm:1.2} follows from the use of the Green tensor representation (\ref{4.40}) in the nonlinear part, which avoids the decomposition of the Leray projection operator required in \cite{2018H}, and yields the decay rate $-\frac{1}{2}$. In the proof of the  lower bound (\ref{lower}), we also utilize the Green tensor representation (\ref{4.40}).

    The rest of the paper is organized as follows: In Section \ref{sec:2}, we introduce some notations and preliminary results. In Section \ref{sec:3}, the $L^1$ asymptotic expansion of the Stokes system is established as that given in Theorem \ref{thm:1.1}.  In Section \ref{sec:4}, we  establish the \( L^1\) profile of the Navier--Stokes system and give the first sufficient condition for \( \boldsymbol{u}(\cdot,t) \notin L^1(\mathbb{R}^n_+) \) for long time behavior, which corresponds to Case (1) in Theorem \ref{thm:1.2}. In Section \ref{sec:5}, we give the pointwise lower bound estimate in Case (2) of Theorem \ref{thm:1.2} . In Section \ref{sec:6}, we  prove Theorem \ref{thm:1.3} and provide the second sufficient condition for \( \boldsymbol{u}(\cdot,t) \notin L^1(\mathbb{R}^n_+) \), which corresponds to short time behavior in Case (3) of Theorem \ref{thm:1.2}.

    Throughout this paper, the constant $C$ may depend on the initial data $\boldsymbol{a}$ and the dimension $n$, but such dependence will not be explicitly indicated.

\section{Notations and Preliminaries }\label{sec:2}
In this section, we introduce some notations and some preliminary results including the Green tensor formula by Solonnikov, existence and pointwise estimates of solutions to (\ref{1.1}).

Let $C_0^\infty(\mathbb{R}^n_+)$ denote the space of smooth real-valued functions with compact support in $\mathbb{R}^n_+$. 
We define the divergence-free test function space
\[
C_{0,\sigma}^\infty(\mathbb{R}^n_+) := \big\{ \boldsymbol{\phi}=(\phi_1,\phi_2,\dots,\phi_n) \in C_0^\infty(\mathbb{R}^n_+) \;\big|\; \nabla \cdot \boldsymbol{\phi} = 0 \big\},
\]
and for $1 < r < \infty$, its closure in $L^r(\mathbb{R}^n_+)$ is denoted by $L_\sigma^r(\mathbb{R}^n_+)$ which coincides with the space of divergence-free vector fields $\boldsymbol{u} \in L^r(\mathbb{R}^n_+)$ satisfying $\boldsymbol{u}\cdot\boldsymbol{n}|_{\partial\mathbb{R}^n_+} =u_n|_{\partial\mathbb{R}^n_+} = 0$ in the weak sense. Consider Helmholtz decomposition (see \cite{1988BM}). 
 $$L^r(\mathbb{R}^n_+)=L^r_\sigma(\mathbb{R}^n_+)\oplus L^r_\pi(\mathbb{R}^n_+),\quad1<r<\infty,$$
 with
$$L^r_\pi(\mathbb{R}^n_+)=\{\nabla p\in L^r(\mathbb{R}^n_+):p\in L^r_{loc}(\overline{\mathbb{R}^n_+})\}.$$
Let $\mathbb{P}$ be the associated  projection operator from $L^r(\mathbb{R}^n_+)$ onto $L^r_\sigma(\mathbb{R}^n_+)$.
The problem (\ref{1.1}) can be written in the form
\begin{equation}
    \partial_t \boldsymbol{u} + \mathbb{A}\boldsymbol{u} = -\mathbb{P}(\boldsymbol{u} \cdot \nabla)\boldsymbol{u} \quad \text{in } \mathbb{R}^n_+ \times (0,\infty), \quad \boldsymbol{u}(x,0) = \boldsymbol{a}(x),
\end{equation}
which is also transformed into the integral equation\begin{equation}\label{duhamel}\boldsymbol{u}(x,t)=e^{-t\mathbb{A}}\boldsymbol{a}-\int_{0}^{t}e^{-(t-s)\mathbb{A}}\mathbb{P}\left(\boldsymbol{u}\cdot\nabla \boldsymbol{u}(y,s)\right)ds\end{equation} where  $\mathbb{A}=-\mathbb{P}\Delta,$ $e^{-t\mathbb{A}}$ is the Stokes semigroup generated by $\mathbb{A}$ and $\boldsymbol{v}=e^{-t\mathbb{A}}\boldsymbol{a}$ is the solution of the Stokes system (\ref{1.2}).

 For $\boldsymbol{a} (x)= \left(a_1(x), a_2(x), \dots, a_n(x)\right): \mathbb{R}_+^n \to \mathbb{R}^n$, $\nabla\cdot \boldsymbol{a}=0,  a_n|_{\partial\mathbb{R}^n_+}=0,x\in\mathbb{R}^n_+$ and $t>0$, the  solution representation formula  of  (\ref{1.2}) by Solonnikov (see \cite{2003S_1}) is given by\begin{equation}\label{3.26}\begin{aligned}\boldsymbol{v}(x,t)&=\int_{\mathbb{R}^{n}_{+}}\mathcal{G}(x,y,t)\boldsymbol{a}(y)\,dy\\&=\int_{\mathbb{R}^{n}_{+}}[G_{t}(x-y)- G_{t}(x-y^{*})]
\boldsymbol{a}(y)\,dy+\int_{\mathbb{R}^{n}_{+}}M(x,y,t)\boldsymbol{a}(y) \,dy.\end{aligned}\end{equation} 
Here  $M(x,y,t)=(M_{ij})_{i,j=1,2,...,n}$ is a matrix defined as
\begin{equation}\label{2.13}M_{ij}(x,y,t)=-4(1-\delta_{nj})\frac{\partial}{\partial x_{j}}\int_{0}^{x_n}\int_{\mathbb{R}^{n-1}}\frac{\partial E(x-z)}{\partial x_{i}}G_{t}(z-y^{*})\,dz,\end{equation} where $y^*=(y_1,y_2,\dots,y_{n-1},-y_n)$.

Let $  k = (k' , k_n),l=(l',l_n)$. Then the following estimate for $M_{i j}(x,y,t)$
 is given by Solonnikov \cite{2003S}. For
$x = (x' , x_n), y = (y', y_n)\in \mathbb{R}^{n}
_{+}, t > 0$ and $i,j=1,2,\dots,n,$ one has
\begin{equation}\label{2.9}|\partial_{x}^{k}\partial_{y}^{l}M_{ij}(x,y,t)|\leq Ct^{-\frac{l_n}{2}}(x_{n}+\sqrt{t})^{-k_{n}}
(|x'-y'|+x_{n}+y_{n}+\sqrt{t})^{-n-|k'|-|l'|}e^{-\frac{cy_{n}^{2}}{t}}.\end{equation}

Then we introduce two lemmata about the structure of Green tensor.
\begin{Lemma}[\hspace{-0.1pt}\cite{2003S} or {\cite[Lemma 4.1]{2024H}}]\label{lem:2.2}
Let $\mathscr{S}(\mathbb{R}^n)$ denote the class of Schwartz functions.
For any $x = (x_1, x_2, \dots, x_n) \in \mathbb{R}^n_+\ (n\geq 2)$  and $g \in \mathscr{S}(\mathbb{R}^n)$, it holds that
\[
\sum_{i=1}^n \partial_{x_i} \int^{x_n}_0 \int_{\mathbb{R}^{n-1}} g(y) \partial_{x_i} E(x - y) \,\,dy = -\frac{1}{2} g(x).
\]
\end{Lemma}

The following lemma is a direct consequence of   \cite[Lemma 4.1]{2018H}. For completeness, we provide a proof here.

\begin{Lemma}\label{lem:2.4}For $x\in \mathbb{R}^{n}_{+}\ (n\geq 2), y_{n}>0,\ t>0$ and $i=1,2,\dots,n$, it holds that
\begin{equation}\label{2.10}
\partial_{x_{n}}M^*_i(x,0',y_n,t)=\partial_{y_n}M^*_i(x,0',y_n,t)-4G_t^{(1)}(y_n)\int_0^\infty \partial_{x_i}[G_{t+\tau}^{(n-1)}(x')G_\tau^{(1)}(x_n)]\,d\tau.\end{equation}
Here $M_i^*(x,0',y_n,t)$ is defined in (\ref{1.6}) and $0'=(0,0,\dots,0)\in \mathbb{R}^{n-1}$. \end{Lemma}
\begin{proof}
   Note that
    $$\partial_{x_i}E(x)=\int_0^\infty \partial_{x_i}G_\tau(x)\,d\tau\quad \text{for $x\in\mathbb{R}^n_+,t>0$ and $i=1,2,\dots,n$}.$$
   Therefore for $x\in \mathbb{R}^n_+,y_n>0,t>0$ and $i<n$, one has
    \begin{equation}\label{2.11}
       \begin{aligned} &M^*_i(x,0',y_n,t)\\=&-4\int_{0}^{x_n}\int_{\mathbb{R}^{n-1}}\frac{\partial E(x-z)}{\partial x_{i}}G_{t}(z',y_n+z_n))\,dz\\=&-4\int_0^{x_n}\int_{\mathbb{R}^{n-1}}\int_0^\infty\partial_{x_i}G_\tau^{(n-1)}(x'-z')G_\tau^{(1)}(x_n-z_n)G_t^{(n-1)}(z')G_t^{(1)}(z_n+y_n)\,d\tau \,dz\\=&-4\int_0^\infty\partial_{x_i}G_{\tau+t}^{(n-1)}(x')\int_0^{x_n}G_\tau^{(1)}(x_n-z_n
)G_t^{(1)}(z_n+y_n)\,dz_n\,d\tau,    \end{aligned}\end{equation}where the properties\[\int_{\mathbb{R}^{n-1}}G^{(n-1)}_\tau(x'-z')G^{(n-1)}_t(z')\,dz'=G^{(n-1)}_{\tau+t}(x')\]has been used.

Furthermore, the straightforward computations give
 \begin{equation}\label{change}
       \begin{aligned} &\partial_{x_n}\int_0^{x_n}G_\tau^{(1)}(x_n-z_n
)G_t^{(1)}(z_n+y_n)\,dz_n\\=&G^{(1)}_\tau (0)G_t^{(1)}(x_n+y_n)-\int_0^{x_n}\partial_{z_n}G_\tau^{(1)}(x_n-z_n)G_t^{(1)}(z_n+y_n)\,dz\\=&G^{(1)}_\tau (0)G_t^{(1)}(z_n+y_n)+\int_0^{x_n}G_\tau^{(1)}(x_n-z_n)\partial_{y_n}G_t^{(1)}(z_n+y_n)\,dz\\&-G^{(1)}_\tau(0)G_t^{(1)}(z_n+y_n)+G^{(1)}_\tau(x_n)G_t^{(1)}(y_n)\\=&\partial_{y_n}\int_0^{x_n}G_\tau^{(1)}(x_n-z_n)G_t^{(1)}(z_n+y_n)\,dz+G^{(1)}_\tau(x_n)G_t^{(1)}(y_n). \end{aligned}\end{equation}
Combining this with 
 (\ref{2.11}), 
 yields Lemma \ref{lem:2.4} for $i<n$. Similarly, (\ref{2.10}) holds for $i=n$.
\end{proof}Using the property $\nabla\cdot \boldsymbol{u}=0$ and (\ref{duhamel}), problem (\ref{1.1}) can be also written as$$\begin{aligned}\boldsymbol{u}(x,t)=e^{-t\mathbb{A}}\boldsymbol{a}-\int_0^t e^{-(t-s)\mathbb{A}}\mathbb{P}\nabla\cdot (\boldsymbol{u}\otimes \boldsymbol{u})\,ds.\end{aligned}$$
We next introduce some analysis for the operator $e^{-(t-s)\mathbb{A}}\mathbb{P}\nabla\cdot$ in the second term on the right-hand side.

 From \cite{2003S_1} or \cite[(2.6)]{2017CJ}, for matrix function $\mathscr{F}(y,s)=(F_{jk}(y,s))_{j,k=1,2,\dots,n}$ with $F_{nk}\big|_{y_n=0}=0$, $x\in \mathbb{R}^n_+$, $t>s>0$ and $i=1,2,\dots,n,$
 \begin{equation}\label{4.40}\begin{aligned}
    &[e^{-(t-s)\mathbb{A}}\mathbb{P}\nabla\cdot \mathscr{F}]_i=-\sum_{j\neq n}\sum_{l=1}^n\int_{\mathbb{R}^n_+}\partial_{y_l}\mathcal{G}_{ij}(x,y,t-s)F_{lj}(y,s)\,dy\\&+\sum_{j\neq n}\int_{\mathbb{R}^n_+}\partial_{y_j}\mathcal{G}_{ij}(x,y,t-s)F_{in}(y,s)\,dy+\sum_{l\neq n}\int_{\mathbb{R}^n_+}\partial_{y_l}\mathcal{G}_{in}(x,y,t-s)F_{nl}(y,s)\,dy\\&-\sum_{k,j,l\neq n}\int_{\mathbb{R}^n_+}\partial_{x_j}\partial_{x_l}K_{ijk}^+(x,y,t-s)F_{lk}(y,s)\,dy-\sum_{j,l\neq n}\int_{\mathbb{R}^n_+} \partial_{x_j}\partial_{x_l}K_{ijn}^-(x,y,t-s)(F_{ln}+F_{nl})(y,s)\,dy\\&+\sum_{j,m\neq n}\int_{\mathbb{R}^n_+} \partial^2_{x_m}K_{ijj}^+(x,y,t-s)F_{nn}(y,s)\,dy-\sum_{k,l\neq n}\int_{\mathbb{R}^n_+} \partial_{x_k}\partial_{x_l}K_{inn}^-(x,y,t-s)F_{lk}(y,s)\,dy\\&+
    \sum_{l,m\neq n}\int_{\mathbb{R}^n_+} \partial^2_{x_m}K_{inl}^+(x,y,t-s)(F_{ln}+F_{nl})(y,s)\,dy+\sum_{m\neq n}\int_{\mathbb{R}^n_+}\partial_{x_m}^2K^-_{inn}(x,y,t-s)F_{nn}(y,s)\,dy,
\end{aligned}\end{equation}
    where $$K^\pm_{ijq}(x,y,t-s)=\int_{\mathbb{R}^n_+}\mathcal{G}_{ij}(x,z,t-s)\partial_{z_q}E^{\pm}(y,z)\,dz,$$ with $$E^+(y,z)=E(y-z) \quad\text{and}\quad E^-(y,z)=E(y-z^*).$$
    Solonnikov \cite{2003S_1} gave the estimates
    \begin{equation}\label{6.47}
        |\partial_{x'}^{m'}K^\pm_{ijq}(x,y,t)|\leq (|x-y|+\sqrt{t})^{-(n-1+|m'|)}.
    \end{equation}
    Crispo and Maremonti \cite{2006CM} improved the estimate as
    \begin{equation}
    \begin{aligned}
        &|\partial_t^k\partial_{x}^{m}\partial_y^{l'} K^\pm_{ijq}(x,y,t)|\leq C(|x-y|+\sqrt{t})^{-(n-1+|m|+|l'|+2k)}\\&\quad+C(x_n+\sqrt{t})^{-m_n}(|x-y|+\sqrt{t})^{-(n-1+|m'|+|l'|)}\left[(x_n+\sqrt{t})^{-m_n}t^{-k}+(|x_n-y_n|+\sqrt{t})^{-(m_n+2k)}\right].
    \end{aligned}\end{equation}We now recall several existence and decay results of solution to problem (\ref{1.1}).\begin{Lemma}[\hspace{-0.1pt}\cite{2001FM}]\label{lem:2.1}
Let the initial data $\boldsymbol{a} \in L^2_\sigma(\mathbb{R}^n_+) \cap L^n(\mathbb{R}^n_+)\ (n \geq 2)$. Then there exist $T > 0$ and a unique strong solution $\boldsymbol{u} \in C([0,T); L^n_\sigma(\mathbb{R}^n_+))$ of (\ref{1.1}). Moreover, there exists an $\varepsilon_0 > 0$ such that if $\|\boldsymbol{a}\|_{L^n(\mathbb{R}^n_+)} \leq \varepsilon_0$ for $n \geq 3$, then the solution exists globally, i.e., $T = +\infty$. For $n = 2$, the solution exists globally without any smallness condition.\end{Lemma}
\begin{Lemma}[\hspace{-0.1pt}\cite{1988BM,2001FM,2012H,2014Hw}]\label{lem:decay}
For $n\geq 2$, assume $\boldsymbol{a} \in L^2_\sigma(\mathbb{R}^n_+) \cap L^1(\mathbb{R}^n_+) \cap L^n(\mathbb{R}^n_+)$.
If, in addition $\|\boldsymbol{a}\|_{L^n(\mathbb{R}^n_+)}$ is sufficiently small in the case $n \geq 3$, then the global strong solution $\boldsymbol{u}$ given in Lemma \ref{lem:2.1} satisfies
\begin{equation}\label{one}
    \|\boldsymbol{u}(t)\|_{L^2(\mathbb{R}^n_+)}\leq CK_1(\boldsymbol{a})(1+t)^{-\frac{n}{4}} \quad\text{ for } t>0,
\end{equation}
where the constant $C$ depends only on $n$, and $K_1(\boldsymbol{a})=\|\boldsymbol{a}\|_{L^1(\mathbb{R}^n_+)}+\|\boldsymbol{a}\|^2_{L^1(\mathbb{R}^n_+)}+\|\boldsymbol{a}\|_{L^2(\mathbb{R}^n_+)}+\|\boldsymbol{a}\|^2_{L^2(\mathbb{R}^n_+)}$.

Moreover, if $x_n\boldsymbol{a}\in L^1(\mathbb{R}^n_+)$, then 
\begin{equation}\label{two}
    \|\boldsymbol{u}(t)\|_{L^2(\mathbb{R}^n_+)}\leq CK_2(\boldsymbol{a})(1+t)^{-\frac{n+2}{4}}\quad\text{ for } t>0,
\end{equation}
where the constant $C$ depends only on $n$, and $K_2(\boldsymbol{a})=K_1(\boldsymbol{a})+\int_{\mathbb{R}^n_+} y_n |\boldsymbol{a}(y)| \, dy + \left( \int_{\mathbb{R}^n_+} y_n |\boldsymbol{a}(y)| \, dy \right)^2$.

Furthermore, if\[
\||x|\boldsymbol{a}\|_{L^2(\mathbb{R}^n_+)} + \|(1+|x|)\nabla \boldsymbol{a}\|_{L^2(\mathbb{R}^n_+)} + \||x|\boldsymbol{a}\|_{L^1(\mathbb{R}^n_+)} < \infty,
\]
then 
\begin{equation}\label{three}
\||x|^\gamma \boldsymbol{u}(t)\|_{L^2(\mathbb{R}^n_+)} \leq C(1+t)^{-\frac{n}{4}+\frac{\gamma}{2}} \quad\text{ for } 0 < \gamma < 1 \text{ and }t>0.
\end{equation}\end{Lemma}\begin{Remark}
The explicit dependence of the constants on the right-hand sides of (\ref{one}) and (\ref{two}) on the initial data $\boldsymbol{a}$ can be derived by tracing the proofs in \cite{1988BM,2001FM}.
\end{Remark}

\begin{Lemma}[\hspace{-0.1pt}\cite{2017CJS,2017CJ}]\label{lem:2.3}
    Assume  $\boldsymbol{a}\in L^2_{\sigma}(\mathbb{R}^n_+)\ (n\geq 2)$ satisfies
   $|\boldsymbol{a}(x)|\leq \frac{C}{(1+|x|)^n}$.
    Then there exists a $T_2>0$ and a unique strong solution $\boldsymbol{u}\in L^\infty(\mathbb{R}^n_+\times(0,T_2))$ of (\ref{1.1}) satisfying
    \begin{equation}
        |\boldsymbol{u}(x,t)| \leq \frac{C}{(1+|x|+\sqrt{t})^{n}}\quad \text{for } t \in (0,T_2).
    \end{equation}
\end{Lemma}
Define the function space with vanishing normal trace as
\begin{equation}\label{2.20}
Z_{\bar{b},b}=\left\{f=(f_1,f_2,\dots,f_n)\in L^\infty_{loc}(\mathbb{R}^n_+)\middle| 
\begin{aligned}
&\|f\|_{Z_{\bar{b},b}} = \sup_{x\in\mathbb{R}^n_+} \frac{|f(x)|(1+|x|)^{\bar{b}}(1+x_n)^b}{x_n^b} < +\infty,\\
&\nabla\cdot f = 0\ \text{and}\ f_n|_{\partial\mathbb{R}^n_+} = 0
\end{aligned}
\right\}.
\end{equation}
 \begin{Lemma}[{\!\!\cite[Theorem 2.3]{2025KLLT}}]\label{lem:2.6}
     Let $n \geq 2, b \in [0, 1]$, and $b< \bar{b} \leq n$. For any initial data $\boldsymbol{a} \in Z_{\bar{b},b}$, there exist $T_3>0$ and a mild solution $\boldsymbol{u} \in L^\infty (0, T_3;Z_{\bar{b},b })$ of (\ref{1.1}). Moreover, the mild solution is unique in
the class $L^\infty(\mathbb{R}^n_+\times (0,T_3))$.
 \end{Lemma}Next, we introduce two technical lemmata which will be used to prove the pointwise estimates.

Recall the definitions of $L_{ij}(x,t)$ in (\ref{L_{ij}}) and the domain $\Lambda_\sigma$ in (\ref{domain}).
\begin{Lemma}\label{lem:L_{ij}}
    For $i=1,2,...,n,j=1,2,\dots,n-1,t>0,x\in\Lambda_{\sigma},\frac{|x|}{\sqrt{t}}>1$ and $\sigma> 1$  large enough, there exist two constants $C, \widetilde{C} > 0$, independent of $t$, such that (\ref{pointwise}) holds.
\end{Lemma}
\begin{proof} \textbf{Step 1. $i=n,j<n$.} For $x\in\mathbb{R}^n_+$, $t>0$ and $\frac{|x|}{\sqrt{t}}\geq1$, one has  $$\begin{aligned}|L_{nj}(x,t)|&=C\int_{0}^{\infty}\frac{x_{n}}{\tau}\frac{|x_{j}|}{t+\tau}\frac{1}{(t+\tau)^{\frac{n-1}{2}}}\frac{1}{\tau^{\frac{1}{2}}}e^{-\frac{|x'|^2}{4(\tau+t)}}e^{-\frac{|x_n|^2}{4\tau}}\,d\tau\\&\geq C\int_{0}^{\infty}\frac{x_n |x_j|}{(\tau+t)^{\frac{n}{2}+2}}e^{-\frac{|x|^2}{4\tau}}\,d\tau\\&\geq C\int_{0}^{\infty}\frac{x_n |x_j|}{|x|^{n+4}}(\frac{\tau+t}{|x|^{2}})^{-(\frac{n}{2}+2)}e^{-\frac{|x|^2}{4\tau}}\,d\tau\\&=C\frac{x_n |x_j|}{|x|^{n+2}}\int_{0}^{\infty}\frac{\lambda^{n+1}}{(\lambda^2+1)^{\frac{n}{2}+2}}e^{\frac{-\lambda^2}{4}}\,d\lambda\\&=C\frac{x_n |x_j|}{|x|^{n+2}}.\end{aligned}$$

  For $x\in\Lambda_{\sigma}$ with $\sigma>1$ and $t>0$, it holds that \begin{equation}\label{4.52}\begin{aligned}|L_{nj}(x,t)|&\leq C\int_{0}^{\infty}\frac{x_n |x_j|}{\tau^{\frac{n}{2}+2}}e^{-\frac{|x'|^2}{4(\tau+t)}}e^{-\frac{x_{n}^{2}}{4\tau}}\,d\tau\\&\leq C\int_{0}^{\infty}\frac{x_n |x_j|}{|x|^{n+4}}(\frac{|x|^2}{\tau})^{\frac{n}{2}+2}e^{-\frac{x^2_{n}}{|x|^2}\frac{|x|^2}{4\tau}}\,d\tau\\&\leq C\frac{x_n |x_j|}{|x|^{n+2}}\int_{0}^{\infty}\lambda^{n+1}e^{-\frac{\sigma^2}{1+\sigma^2}\frac{\lambda^2}{4}}\,d\lambda\\&\leq \widetilde{C}\frac{x_n |x_j|}{|x|^{n+2}}.\end{aligned}\end{equation}

\textbf{Step 2. $i,j<n$ and $j\neq i$.}  For $x\in \mathbb{R}^n_+$, $t>0$ and $\frac{|x|}{\sqrt{t}}\geq 1$, one has 
 \begin{equation}\begin{aligned}|L_{ ij}(x,t)|&=C\int_{0}^{\infty}\frac{|x_{i}|}{t+\tau}\frac{|x_{j}|}{t+\tau}\frac{1}{(t+\tau)^{\frac{n-1}{2}}}\frac{1}{\tau^{\frac{1}{2}}}e^{-\frac{|x'|^2}{4(\tau+t)}}e^{-\frac{|x_n|^2}{4\tau}}\,d\tau\\&\geq C\int_{0}^{\infty}\frac{|x_i x_j|}{(\tau+t)^{\frac{n}{2}+2}}e^{-\frac{|x|^2}{4\tau}}\,d\tau\\&\geq C\int_{0}^{\infty}\frac{|x_i x_j|}{|x|^{n+4}}(\frac{\tau+t}{|x|^{2}})^{-(\frac{n}{2}+2)}e^{-\frac{|x|^2}{4\tau}}\,d\tau\\&\geq C\frac{|x_i x_j|}{|x|^{n+2}}\int_{0}^{\infty}\frac{\lambda^{n+1}}{(\lambda^2+1)^{\frac{n}{2}+2}}e^{-\frac{\lambda^2}{4}}\,d\lambda\\&=C\frac{|x_i x_j|}{|x|^{n+2}}.\end{aligned}.\end{equation}
 For $x\in \Lambda_{\sigma}$ with $\sigma>1$ and $t>0$, it holds that\begin{equation}\begin{aligned}|L_{ij}(x,t)|&\leq C\int_{0}^{\infty}\frac{|x_i x_j|}{\tau^{\frac{n}{2}+2}}e^{-\frac{|x'|^2}{4(\tau+t)}}e^{-\frac{x_{n}^{2}}{4\tau}}\,d\tau\\&\leq C\int_{0}^{\infty}\frac{|x_i x_j|}{|x|^{n+4}}(\frac{|x|}{\sqrt{\tau}})^{n+4}e^{-\frac{|x'|^2}{4(\tau+t)}}e^{-\frac{x^2_{n}}{|x|^2}\frac{|x|^2}{4\tau}}\,d\tau\\&\leq C\int_{0}^{\infty}\frac{|x_i x_j|}{|x|^{n+4}}\lambda^{n+1}e^{-\frac{\sigma^2}{1+\sigma^2}\frac{\lambda^2}{4}}\,d\lambda\\&\leq \widetilde{C}\frac{|x_i x_j|}{|x|^{n+2}}.\end{aligned}\end{equation}

 \textbf{Step 3. $i\leq n-1$ and $j= i$.} Note that
 $$\begin{aligned}L_{ii}(x,t)&= C\int_{0}^{\infty}\frac{x_i x_i}{(\tau+t)^{\frac{n-1}{2}+2}\tau^{\frac{1}{2}}}e^{-\frac{|x'|^2}{4(\tau+t)}}e^{-\frac{x_{n}^{2}}{4\tau}}\,d\tau- 2C\int_{0}^{\infty}\frac{1}{(\tau+t)^{\frac{n-1}{2}+1}\tau^{\frac{1}{2}}}e^{-\frac{|x'|^2}{4(\tau+t)}}e^{-\frac{x_{n}^{2}}{4\tau}}\,d\tau.\end{aligned}$$  For $x\in\Lambda_\sigma$ with $\sigma>1$, $t>0$ and $\frac{|x|}{\sqrt{t}}\geq1$, we have
 \begin{equation}\label{C_1}
     C\frac{|x_ix_i|}{|x|^{n+2}}\leq\left|\int_{0}^{\infty}\frac{x_i x_i}{(\tau+t)^{\frac{n-1}{2}+2}\tau^{\frac{1}{2}}}e^{-\frac{|x'|^2}{4(\tau+t)}}e^{-\frac{x_{n}^{2}}{4\tau}}\,d\tau\right|\leq C_1\frac{|x_ix_i|}{|x|^{n+2}}
 \end{equation}and
 \begin{equation}\label{4.66}
     C_2\frac{1}{|x|^n}\leq\left|\int_{0}^{\infty}\frac{1}{(\tau+t)^{\frac{n-1}{2}+1}\tau^{\frac{1}{2}}}e^{-\frac{|x'|^2}{4(\tau+t)}}e^{-\frac{x_{n}^{2}}{4\tau}}\,d\tau\right|\leq C\frac{1}{|x|^n}.
 \end{equation}
 For $\sigma > \sqrt{\frac{C_1}{C_2}}$, it follows from (\ref{C_1}), (\ref{4.66}) and $x_n > \sigma \lvert x' \rvert$ that there exist $C,\widetilde{C}>0$ such that
 \[
  C\frac{1}{|x|^n}\leq\left|L_{ii}(x,t)\right|\leq \widetilde{C}\frac{1}{|x|^n}.
 \]
 Hence the proof of the lemma is completed.\end{proof}
Next, we have the following elementary lemma for integrals.
\begin{Lemma}\label{lem:2.5}
  For $n\geq2,n+1\leq b<n+2,$  $x\in\mathbb{R}^n_+$ and $t>0$ \begin{equation}\label{2.14}
         \int_0^t\int_{\mathbb{R}^n_+}(|x-y|+\sqrt{t-s})^{-(n+1)}(1+|y|+\sqrt{s})^{-b}\,dy\,ds\leq C(1+|x|+\sqrt{t})^{-(n+1)}\frac{t^{\frac{1}{2}}}{(1+t)^{\frac{b-n-1}{2}}}.
     \end{equation}
     \begin{proof}
         First, one has\begin{equation}\label{2.15}
            \int_0^\frac{t}{2}\int_{\mathbb{R}^n_+}(|x-y|+\sqrt{t-s})^{-(n+1)}(1+|y|+\sqrt{s})^{-b}\,dy\,ds\leq \int_0^\frac{t}{2}(t-s)^{-\frac{1}{2}}\,ds\leq Ct^{\frac{1}{2}}
         \end{equation}
         and
         \begin{equation}\label{2.16}
             \int_\frac{t}{2}^t\int_{\mathbb{R}^n_+}(|x-y|+\sqrt{t-s})^{-(n+1)}(1+|y|+\sqrt{s})^{-b}\,dy\,ds\leq C\frac{t^{\frac{1}{2}}}{(1+t)^{\frac{b}{2}}}.
         \end{equation}
        Let us divide $\mathbb{R}^n_+$ into two domains \[D_1=\left\{y\in\mathbb{R}^n_+:|y|\leq\frac{|x|+\sqrt{t-s}}{2}\right\}\quad \text{and}\quad D_2=\left\{y\in\mathbb{R}^n_+:|y|>\frac{|x|+\sqrt{t-s}}{2}\right\}.\]
         One has\begin{equation}
             \int_{D_1}(|x-y|+\sqrt{t-s})^{-(n+1)}(1+|y|+\sqrt{s})^{-b}\,dy\leq C(|x|+\sqrt{t-s})^{-(n+1)}(1+\sqrt{s})^{n-b}
         \end{equation}
         and
         \begin{equation}
             \int_{D_2}(|x-y|+\sqrt{t-s})^{-(n+1)}(1+|y|+\sqrt{s})^{-b}\,dy\leq C(1+|x|+\sqrt{t})^{-b}(t-s)^{-\frac{1}{2}}.
         \end{equation}
         If $|x|+\sqrt{t}>1$, it holds that
         \begin{equation}\label{2.19}
             \int_0^\frac{t}{2}\int_{\mathbb{R}^n_+}(|x-y|+\sqrt{t-s})^{-(n+1)}(1+|y|+\sqrt{s})^{-b}\,dy\,ds\leq C(|x|+\sqrt{t})^{-(n+1)}\frac{t^{\frac{1}{2}}}{(1+t)^{\frac{b-n-1}{2}}}.
         \end{equation}
         If $|x|>1+\sqrt{t},$ one has
         \begin{equation}\label{2.29}
             \int_\frac{t}{2}^t\int_{\mathbb{R}^n_+}(|x-y|+\sqrt{t-s})^{-(n+1)}(1+|y|+\sqrt{s})^{-b}\,dy\,ds\leq C|x|^{-(n+1)}\frac{t^{\frac{1}{2}}}{(1+t)^{\frac{b-n-1}{2}}}.
         \end{equation}
         Combining (\ref{2.15}), (\ref{2.16}), (\ref{2.19}), and (\ref{2.29}) yields (\ref{2.14}).
     \end{proof}
 \end{Lemma}
 At the end of this section, we present a basic cancellation property for the initial data of problem (\ref{1.1}), which plays a crucial role in the proof of Theorem \ref{thm:1.1}.

 Let $\overline{\mathbb{R}^n_+} = \{(x_1,\ldots,x_n) \in \mathbb{R}^n \mid x_n \geq 0\}$ denote the closed upper half space. The space of smooth functions with compact support in $\overline{\mathbb{R}^n_+}$ is denoted by
\[
C_c^\infty(\overline{\mathbb{R}^n_+}) := \big\{ f: \overline{\mathbb{R}^n_+} \to \mathbb{R} \;\big|\; f \in C^\infty(\overline{\mathbb{R}^n_+}),\ \operatorname{supp}(f) \subset \overline{\mathbb{R}^n_+} \text{ is compact} \big\}.
\]
 \begin{Lemma}[Normal Flux Conservation]\label{A.1}
Let $\boldsymbol{a} = (a_1,a_2, \dots, a_n) \in  L^1(\mathbb{R}^n_+)$ be a vector field satisfying $\nabla \cdot \boldsymbol{a} = 0$ and $a_n|_{x_n=0} = 0$ in the sense of distribution, that is for any $\phi(x)\in C_c^\infty(\overline{\mathbb{R}^n_+}),$
\[\int_{\mathbb{R}^n_+} \boldsymbol{a}(x)\nabla\phi(x)\,dx=0.\]
Then for almost every $x_n\geq 0$, the normal flux satisfies
\[
\int_{\mathbb{R}^{n-1}} a_n(x',x_n) \,dx' = 0.
\]
\end{Lemma}

\begin{proof}
We proceed in two steps.

 \textbf{Step 1. Truncation function construction.} 
For any $R > 0$, let $\chi_R \in C_c^\infty(\mathbb{R}^{n-1})$ satisfy
\[
\chi_R(x') = \begin{cases}
1 & \text{for } |x'| \leq R, \\
0 & \text{for } |x'| \geq 2R,
\end{cases}
\quad \text{with } \|\nabla \chi_R\|_{L^\infty} \leq \frac{C}{R}.
\]
Let $\psi$ be defined by
\[
\psi(t) :=
\begin{cases}
C e^{-\frac{1}{1-t^2}} & t<1, \\
0 & t\geq 1,
\end{cases}
\]
where $C$ is the normalization constant ensuring $\int_{-1}^1 \psi(t)dt = 1$.
For any $\alpha>0$ and $\epsilon > 0$, define $\eta_\epsilon \in C_c^\infty(\mathbb{R})$ as
\[
\eta_\epsilon(y) = 
\begin{cases}
1 & 0 \leq y \leq \alpha-\epsilon, \\
1 - \int_{-1}^{(y-\alpha)/\epsilon} \psi(t) dt & \alpha-\epsilon < y < \alpha+\epsilon, \\
0 & y \geq \alpha+\epsilon.
\end{cases}
\]

Clearly, $\eta_\epsilon(y)$ satisfies

\[
\eta_\epsilon'(y) = -\frac{1}{\epsilon}\psi(\frac{y-\alpha}{\epsilon}).
\]

Therefore, as $\epsilon\to 0^+$, $\eta_\epsilon'$ converges  to $ -\delta(y-\alpha)$ in the the sense of distribution.

Consider test functions $\phi(x',x_n) = \chi_R(x')\eta_\epsilon(x_n)$. The distributional divergence condition yields that for $R>0$ and $\epsilon>0$, one has
\[
\int_0^\infty \int_{\mathbb{R}^{n-1}} \left( \sum_{i=1}^{n-1} a_i \frac{\partial \chi_R}{\partial x_i} \eta_\epsilon + a_n \chi_R \eta_\epsilon' \right) \,dx' \,dx_n = 0.
\]

 \textbf{Step 2. Limit analysis.} 
For the tangential terms
\[
\left| \int_{\mathbb{R}^n_+} a_i \frac{\partial \chi_R}{\partial x_i} \eta_\epsilon \,dx' \,dx_n \right| \leq \frac{C}{R} \|a_i\|_{L^1} \to 0 \quad \text{as } R \to \infty.
\]
For the normal term, we first observe that for almost every $\alpha\in(0,+\infty)$
\[\begin{aligned}
\lim_{\epsilon\to 0}\int_0^\infty \int_{\mathbb{R}^{n-1}} a_n(x)\chi_R(x') \eta_\epsilon'(x_n) \,dx'\,dx_n &= -\lim_{\epsilon\to 0}\int_0^\infty \int_{\mathbb{R}^{n-1}} a_n(x)\chi_R(x')\frac{1}{\epsilon}\psi(\frac{y-\alpha}{\epsilon}) \,dx'\,dx_n\\&=-\int_{\mathbb{R}^{n-1}} \chi_R(x') a_n(x',\alpha)\,dx'.
\end{aligned}\]

By dominated convergence theorem, taking $R \to \infty$ and noting that $\chi_R \to 1$ yields
for almost every $\alpha\in(0,+\infty)$ \[\begin{aligned}
\lim_{R\to +\infty}\lim_{\epsilon \to 0} \int_0^\infty \int_{\mathbb{R}^{n-1}} a_n\chi_R \eta_\epsilon' \,dx'\,dx_n &=-\lim_{R\to +\infty}\int_{\mathbb{R}^{n-1}} \chi_R(x') a_n(x',\alpha)\,dx' \\&=-\int_{\mathbb{R}^{n-1}} a_n(x',\alpha)  \,dx'. \end{aligned}
\]

 This means
\[
\int_{\mathbb{R}^{n-1}} a_n(x',x_n)\,dx' = 0 \quad \text{for almost every } x_n\in(0,+\infty).
\]Hence the proof of the lemma is completed.
\end{proof}
 
\section{$L^1$ profile of the Stokes system}\label{sec:3}In this section, we investigate the behavior of the solution to  (\ref{1.2}), identify the leading term of solution in $L^1(\mathbb{R}^n_+)$ and give the proof of Theorem \ref{thm:1.1}.\begin{proof}[Proof of Theorem \ref{thm:1.1}]Recall (\ref{3.26}), the solution  of (\ref{1.2}) can be represented as follows.\begin{equation}\label{3.26}\begin{aligned}\boldsymbol{v}(x,t)&=\int_{\mathbb{R}^{n}_{+}}\mathcal{G}(x,y,t)\boldsymbol{a}(y)\,dy\\&=\int_{\mathbb{R}^{n}_{+}}[G_{t}(x-y)- G_{t}(x-y^{*})]
\boldsymbol{a}(y)\,dy+\int_{\mathbb{R}^{n}_{+}}M(x,y,t)\boldsymbol{a}(y)\,dy\\&:=\boldsymbol{v}^{(1)}(x,t)+\boldsymbol{v}^{(2)}(x,t).\end{aligned}\end{equation} The rest of the proof is divided into five steps.

\textbf{Step 1. The estimate of $\boldsymbol{v}^{(1)}(x,t)$.} We first show that  $\boldsymbol{v}^{(1)}(\cdot,t)\in L^{1}(\mathbb{R}^n_+)$. Indeed, it holds that
$$\begin{aligned}
    \Vert \boldsymbol{v}^{(1)}(x,t)\Vert_{L^1(\mathbb{R}^n_+)}&\leq\left\Vert\int_{\mathbb{R}^n_+}\left(\int_{-1}^1\partial_{x_n}G_{t}(x'-y',x_n-\sigma y_n)\,d\sigma\right) y_n \boldsymbol{a}(y) \,dy\right\Vert_{L^1(\mathbb{R}^n_+)}\\&\leq \int_{\mathbb{R}^n_+}\left(\int_{-1}^1\Vert\partial_{x_n}G_{t}(x'-y',x_n-\sigma y_n)\Vert_{L^1(\mathbb{R}^n_+)}\,d\sigma\right) y_n |\boldsymbol{a}(y)| \,dy\\&\leq Ct^{-\frac{1}{2}}\int_{\mathbb{R}^n_+}y_n| \boldsymbol{a}(y)|\,dy.
\end{aligned}$$

\textbf{Step 2. Extraction of leading term in $\boldsymbol{v}^{(2)}(x,t)$.} Recall the definition of $M^*_i(x,y,t)$ and $M_{ij}(x,y,t)$ in (\ref{1.6}) and (\ref{2.13}) respectively,  for $i=1,2,\dots,n,x\in\mathbb{R}^n_+$ and $t>0$, we have
\begin{equation}\label{3.27}\begin{aligned}v_{i}^{(2)}(x,t)&=-\sum_{j=1}^{n-1}\int_{\mathbb{R}^{n}_{+}}4(1-\delta_{nj})\partial_{x_j}\left[\int_{0}^{x_n}\int_{\mathbb{R}^{n-1}}\partial_{x_i} E(x-z)G_{t}(z-y^{*})\,dz\right]a_{j}(y)\,dy\\&=\sum_{j=1}^{n-1}\int_{\mathbb{R}^{n}_{+}}4(1-\delta_{nj})\partial_{y_{j}}\left[\int_{0}^{x_n}\int_{\mathbb{R}^{n-1}}\partial_{z_i}E(z)G_{t}(z-(x-y^{*}))\,dz\right]a_j(y)\,dy\\&=-\int_{\mathbb{R}^{n}_{+}}4\left[\int_{0}^{x_n}\int_{\mathbb{R}^{n-1}}\partial_{z_i}E(z)G_{t}(z-(x-y^{*}))\,dz\right]\sum_{j=1}^{n-1}\ \partial_{y_j}a_{j}(y)\,dy\\&=\int_{\mathbb{R}^{n}_{+}}4\left[\int_{0}^{x_n}\int_{\mathbb{R}^{n-1}} \partial_{z_i}E(z)G_{t}(z-(x-y^{*}))\,dz\right]\partial_{y_n}a_{n}(y)\,dy\\&=-\int_{\mathbb{R}^{n}_{+}}M^*_i(x,y,t)\partial_{y_n}a_{n}(y)\,dy\\&=\int_{\mathbb{R}^{n}_{+}}\partial_{y_n}M^*_i(x,y,t)a_{n}(y)\,dy,\end{aligned}.\end{equation}
where the property $\nabla\cdot \boldsymbol{a}=0$ has been used in the fourth equality and $a_n(y)|_{\partial\mathbb{R}^n_+}=0$ is needed in the last equality.

Applying Lemma \ref{A.1} yields that for almost every $y_n>0$,$$\int_{\mathbb{R}^{n-1}}a_n(y',y_n)\,dy'=0.$$
Note that $v_{i}^{(2)}$ can be calculated as follows\begin{equation}\label{3.13}\begin{aligned}v_{i}^{(2)}&=\int_{\mathbb{R}^n_+}\partial_{y_n}M^*_i(x,y,t)a_n(y)\,dy\\&=\int_{\mathbb{R}_+^n}\partial _{y_n}[M^*_i(x',x_n,y',y_n,t)-M^*_i(x',x_n,0',y_n,t)]a_n(y',y_n)\,dy\\&=\int_{\mathbb{R}^{n}_+}(\nabla_{y'}'\partial _{y_n}M^*_i(x',x_n,0',y_n,t)\cdot y' )a_n(y',y_n)\,dy'\,dy_n\\&\quad+\int_{\mathbb{R}^n_+}\partial_{y_n}\left(\int_0^1y'\nabla^2_{\theta y'}M^*_i(x,\theta y',y_n,t)y'^T(1-\theta)\,d\theta\right) a_n(y',y_n)\,dy'\,dy_n\\&=-\int_{\mathbb{R}^{n}_+}(\nabla_{x'}'\partial _{y_n}M^*_i(x',x_n,0',y_n,t)\cdot y' )a_n(y',y_n)\,dy'\,dy_n\\&\quad+\int_{\mathbb{R}^n_+}\partial_{y_n}\left(\int_0^1y'\nabla^2_{x'}M^*_i(x,\theta y',y_n,t)y'^T(1-\theta)\,d\theta\right) a_n(y',y_n)\,dy'\,dy_n,\end{aligned}\end{equation}
where we have used $\int_{\mathbb{R}^{n-1}}a_n(y',y_n)\,dy'=0$ for almost every $y_n>0$ in the first equality and $$\partial_{y_j}M^*_i(x,y,t)=-\partial_{x_j}M_i^*(x,y,t)\quad \text{for } j=1,2,\dots,n-1$$ in the last equality.

It follows from Lemma \ref{lem:2.4} that for $x\in\mathbb{R}^n_+$ and $t>0,$ one has
\begin{equation}\begin{aligned}\label{3.14}&\int_{\mathbb{R}^{n}_{+}}(\nabla_{x'}'\partial _{y_n}M^*_i(x',x_n,0',y_n,t)\cdot y' )a_n(y',y_n)\,dy'\,dy_n\\
= &\int_{\mathbb{R}^{n}_{+}}(\nabla_{x'}'\partial _{x_n}M^*_i(x',x_n,0',y_n,t)\cdot y' )a_n(y',y_n)\,dy'\,dy_n\\
&+\sum_{j=1}^{n-1}4\partial_{x_j}\int_0^\infty\partial_{x_i}[G_{t+\tau}^{(n-1)}(x)G_\tau^{(1)}(x_n)]
\,d\tau\int_{\mathbb{R}^n_+}G_t^{(1)}(y_n)y_j a_n(y',y_n)\,dy'\,dy_n.\end{aligned}\end{equation}
 Recall the pointwise estimate (\ref{2.9}) of $M_{ij}(x,y,t)$, the first term on the right-hand side in (\ref{3.14}) can be estimated as follows
\begin{equation}\label{yuxiang1}\begin{aligned}&\left\Vert\int_{\mathbb{R}^{n}_{+}}(\nabla_{x'}'\partial _{x_n}M^*_i(x',x_n,0',y_n,t)\cdot y' )a_n(y',y_n)\,dy'\,dy_n\right\Vert_{L^1(\mathbb{R}^n_+)}\\\leq &C\int_{\mathbb{R}^{n}_{+}}\left(\int_{\mathbb{R}^n_+}(x_n+\sqrt{t})^{-1}(|x'| + x_n + y_n + \sqrt t )^{-n}\,dx\right)|y'||a_n(y)|\,dy\\\leq &Ct^{-\frac{1}{2}}\int_{\mathbb{R}^n_+}|y'||a_n(y)|\,dy.\end{aligned}\end{equation}
Similarly, one has\begin{equation}\label{yuxiang2}\begin{aligned}&\left\Vert\int_{\mathbb{R}^n_+}\left(\int_0^1y'\nabla^2_{x'}\partial_{y_n}M^*_i(x,\theta y',y_n,t)y'^T(1-\theta)\,d\theta\right) a_n(y',y_n)\,dy'\,dy_n\right\Vert_{L^1(\mathbb{R}^n_+)}\\\leq &C\int_{\mathbb{R}^{n}_{+}}\left(\int_{\mathbb{R}^n_+}t^{-\frac{1}{2}}\left(\int_0^1(|x'-\theta y'| + x_n + y_n + \sqrt{t} )^{-(n+1)}(1-\theta)\,d\theta\right) \,dx\right)|y'|^2|a_n(y)|\,dy\\\leq &Ct^{-1}\int_{\mathbb{R}^n_+}|y'|^2|a_n(y)|\,dy.\end{aligned}\end{equation}Combining (\ref{3.13})--(\ref{yuxiang2}) yields (\ref{linear}).

\textbf{Step 3. The second-order terms from $\boldsymbol{v}^{(1)}(x,t)$.} To prove (\ref{high order}), we first consider $\boldsymbol{v}^{(1)}(x,t)$. The direct calculations give\\\begin{equation}\label{v1}
    \begin{aligned}
        &\boldsymbol{v}^{(1)}(x,t)\\=&\int_{\mathbb{R}^n_+}[G_t(x-y)-G_t(x-y^*)]\boldsymbol{a}(y)\,dy\\=&\int_{\mathbb{R}^n_+}\left(\int_{-1}^1-\partial_{x_n}G_{t}(x'-y',x_n-\theta y_n)\,d\theta\right) y_n \boldsymbol{a}(y) \,dy\\=&-2\partial_{x_n}G_t(x)\int_{\mathbb{R}^n_+}y_n\boldsymbol{a}(y)\,dy-\int_{\mathbb{R}^n_+}[G_t(x'-y')\int_{-1}^{1}[\partial_{x_n}G_t(x_n-\theta y_n)-\partial_{x_n}G_t(x_n)]y_n\boldsymbol{a}(y)\,dy\,d\theta\\&\quad-2\partial_{x_n} G_t(x_n)\int_{\mathbb{R}^n_+}y_n[G_t(x'-y')-G_t(x')]\boldsymbol{a}(y)\,dy\\=& -2\partial_{x_n}G_t(x)\int_{\mathbb{R}^n_+}y_n\boldsymbol{a}(y)\,dy-\int_{-1}^1\int_{\mathbb{R}^n_+}y_nG_t(x'-y')[\partial_{x_n}G_t(x_n-y_n\theta)-\partial_{x_n}G_t(x_n)]\boldsymbol{a}(y)\,dy\,d\theta\\&\quad-2\partial_{x_n} G_t(x_n)\int_{\mathbb{R}^n_+}y_n[G_t(x'-y')-G_t(x')]\boldsymbol{a}(y)\,dy.
    \end{aligned}
\end{equation}\\

Applying the dominated convergence theorem  yields that   for fixed $y_n>0$ and $\theta\in [-1,1]$, one has\begin{equation}\begin{aligned}\label{fix1}
    &\lim_{t\rightarrow+\infty}t^{\frac{1}{2}}\left\|\partial_{x_n}G_t(x_n-\theta y_n)-\partial_{x_n}G_t(x_n)\right\|_{L^1_{x_n}(0,+\infty)}\\=&\lim_{t\rightarrow+\infty}t^{\frac{1}{2}}\left\Vert \partial_{x_n}G_1(\cdot+\theta y_nt^{-\frac{1}{2}})-\partial_{x_n}G_1(\cdot)\right\Vert_{L^1(0,+\infty)}=0
\end{aligned}\end{equation}\\
and  
 \begin{equation}\begin{aligned}\label{fix2}
    &\lim_{t\rightarrow+\infty}\left\|G_t(x'-y')-G_t(x')\right\|_{L^1_{x'}(\mathbb{R}^{n-1})}\\=&\lim_{t\rightarrow+\infty}\left\Vert G_1(\cdot-y't^{-\frac{1}{2}})-G_1(\cdot)\right\Vert_{L^1(\mathbb{R}^{n-1})}=0 \quad\text{ for fixed }  y'\in \mathbb{R}^{n-1}.
\end{aligned}\end{equation}
Then from (\ref{v1})--(\ref{fix2}), one has 
\begin{equation}\label{v^1}
    \begin{aligned}
        &\lim_{t\rightarrow+\infty}t^{\frac{1}{2}}\|\boldsymbol{v}^{(1)}(x,t)+2\partial_{x_n}G_t(x)\int_{\mathbb{R}^n_+}y_n \boldsymbol{a}(y)\,dy\|_{L^1_x(\mathbb{R}^n_+)}\\\leq&\lim_{t\rightarrow+\infty}t^{\frac{1}{2}}\left\|\int_{-1}^1\int_{\mathbb{R}^n_+}G_t(x'-y')[\partial_{x_n}G_t(x_n-y_n\theta)-\partial_{x_n}G_t(x_n)]y_n\boldsymbol{a}(y)\,dy\,d\theta\right\|_{L_x^1(\mathbb{R}^n_+)}\\&+2\lim_{t\rightarrow+\infty}t^{\frac{1}{2}}\left\|\partial_{x_n}G_t(x_n)\int_{\mathbb{R}^n_+}[G_t(x'-y')-G_t(x')]y_n\boldsymbol{a}(y)\,dy\right\|_{L_x^1(\mathbb{R}^n_+)} \\\leq&\lim_{t\rightarrow+\infty}t^{\frac{1}{2}}\int_{-1}^1\int_{\mathbb{R}^n_+}\|G_t(x'-y')\|_{L^1_{x'}(\mathbb{R}^{n-1})}\|\partial_{x_n}G_t(x_n-y_n\theta)-\partial_{x_n}G_t(x_n)\|_{L^1_{x_n}(0,+\infty)}y_n|\boldsymbol{a}(y)|\,dy\,d\theta\\&+2\lim_{t\rightarrow+\infty}t^{\frac{1}{2}}\|\partial_{x_n}G_t(x_n)\|_{L_{x_n}^1(0,+\infty)}\int_{\mathbb{R}^n_+}\|G_t(x'-y')-G_t(x')\|_{L^1_{x'}(\mathbb{R}^{n-1})}y_n|\boldsymbol{a}(y)|\,dy\\\leq&\lim_{t\rightarrow+\infty}\int_{-1}^1\int_{\mathbb{R}^n_+}t^{\frac{1}{2}}\|\partial_{x_n}G_t(x_n-y_n\theta)-\partial_{x_n}G_t(x_n)\|_{L^1_{x_n}(0,+\infty)}y_n|\boldsymbol{a}(y)|\,dy\,d\theta\\&+\lim_{t\rightarrow+\infty}C\int_{\mathbb{R}^n_+}\|G_t(x'-y')-G_t(x')\|_{L^1_{x'}(\mathbb{R}^{n-1})}y_n|\boldsymbol{a}(y)|\,dy\\=&0.
    \end{aligned}
\end{equation}

\textbf{Step 4.  The second-order terms from $\boldsymbol{v}^{(2)}(x,t)$.}  We next  consider the second term on the ride-hand side of \begin{equation}\label{chafen2}\begin{aligned}&\int_{\mathbb{R}^n_+}\partial_{x_n}\nabla_{x'}'M_i^*(x,0',y_n,t)\cdot y'a_n(y)\,dy\\=&\int_{\mathbb{R}^n_+}\partial_{x_n}\nabla_{x'}'M_i^*(x,0,t)\cdot y'a_n(y)\,dy+\int_{\mathbb{R}^n_+}[\partial_{x_n}\nabla_{x'}'M_i^*(x,0',y_n,t)-\partial_{x_n}\nabla_{x'}'M_i^*(x,0,t)]\cdot y'a_n(y)\,dy.\end{aligned}\end{equation}
If $i<n$,
 adapting the argument used in (\ref{2.11}) yields that for $t>0$, one has 
\begin{equation}\label{chafen}
    \begin{aligned}
        &\left\Vert \partial_{x_n}\nabla_{x'}'M_i^*(x,0',y_n,t)-\partial_{x_n}\nabla_{x'}'M_i^*(x,0,t)\right\Vert_{L^1_x(\mathbb{R}^n_+)}\\=&4\left\Vert\int_0^\infty \nabla_{x'}'\partial_{x_i}G_{t+\tau}^{(n-1)}(x')\partial_{x_n}\int_0^{x_n}G_\tau^{(1)}(x_n-z_n)[G_t(z_n+y_n)-G_t(z_n)]\,dz_n\,d\tau\right\Vert_{L^1_x(\mathbb{R}^n_+)}\\\leq
        &C\int_0^\infty(t+\tau)^{-1}\int_0^\infty\bigl|G_\tau^{(1)}(0)[G_t(x_n+y_n)-G_t(x_n)]\\&+\int_0^{x_n}\partial_{x_n}G_{\tau}^{(1)}(x_n-z_n)[G_t(z_n+y_n)-G_t(z_n)]\,dz_n\bigr|\,dx_n\,d\tau\\\leq &C\int_{0}^\infty(t+\tau)^{-1}\tau^{-\frac{1}{2}}\,d\tau\int_0^\infty|G_t(x_n+y_n)-G_t(x_n)|\,dx_n\\&+C\int_0^\infty(t+\tau)^{-1}\left(\int_0^\infty\int_{z_n}^\infty|\partial_{x_n}G_\tau^{(1)}(x_n-z_n)||G_t(z_n+y_n)-G_t(z_n)|\,dx_n\,dz_n\right)\,d\tau\\=&C\int_0^\infty(t+\tau)^{-1}\tau^{-\frac{1}{2}}\,d\tau\int_{0}^\infty |G_t(x_n+y_n)-G_t(x_n)|\,dx_n\\\leq &Ct^{-\frac{1}{2}}\int_{0}^\infty |G_t(x_n+y_n)-G_t(x_n)|\,dx_n.
    \end{aligned}
\end{equation}
With the aid of  the dominated convergence theorem,  for fixed $y_n>0$,  it holds  that \begin{equation}\begin{aligned}\label{3.17}
    \lim_{t\rightarrow+\infty}\int_0^\infty|G_t(x_n+y_n)-G_t(x_n)|\,dx_n=\lim_{t\rightarrow+\infty}\Vert G_1(\cdot+y_nt^{-\frac{1}{2}})-G_1(\cdot)\Vert_{L^1(0,+\infty)}=0.
\end{aligned}\end{equation}

Combining (\ref{chafen}) with (\ref{3.17}) gives
\begin{equation}\label{sl1}
    \begin{aligned}
        &\lim_{t\rightarrow +\infty}t^{\frac{1}{2}}\left\Vert \int_{\mathbb{R}^n_+}[\partial_{x_n}\nabla_{x'}'M_i^*(x,0',y_n,t)-\partial_{x_n}\nabla_{x'}'M_i^*(x,0,t)]\cdot y'a_n(y)\,dy\right\Vert_{L^1_x(\mathbb{R}^n_+)}\\\leq &\lim_{t\rightarrow +\infty}t^{\frac{1}{2}}\int_{\mathbb{R}^n_+}\Vert\partial_{x_n}\nabla_{x'}'M_i^*(x,0',y_n,t)-\partial_{x_n}\nabla_{x'}'M_i^*(x,0,t)\Vert_{L_x^1(\mathbb{R}^n_+)}|y'a_n(y)|\,dy\\\leq &C\lim_{t\rightarrow +\infty}\int_{\mathbb{R}^n_+}\int_0^\infty| G_t^{(1)}(x_n+y_n)-G_t^{(1)}(x_n)|\,dx_n |y'a_n(y)|\,dy\\=&0,
    \end{aligned}
\end{equation}
where the dominated convergence theorem has been used in the last equality.

For $i=n,$  Lemma \ref{lem:2.2} implies
\begin{equation}\label{3.35}\sum_{i=1}^{n}\partial_{x_i}M_i^*(x,y,t)=2G_t(x-y^*).\end{equation} Then using (\ref{3.35}) and  adapting the argument used in (\ref{2.11}) yields that for $t>0$, one has
\begin{equation}
    \begin{aligned}
        &\left\Vert\partial_{x_n}\nabla_{x'}'M_n^*(x,0',y_n,t)-\partial_{x_n}\nabla_{x'}'M_n^*(x,0,t)\right\Vert_{L^1_x(\mathbb{R}^n_+)}\\
        = &4\bigl\Vert\sum_{j=1}^{n-1}\int_0^\infty \nabla_{x'}'\partial_{x_j}^2G_{t+\tau}^{(n-1)}(x')\int_0^{x_n}G_\tau^{(1)}(x_n-z_n)[G_t(z_n+y_n)-G_t(z_n)]\,dz_n\\
        &-\frac{1}{2}\nabla_{x'}'[G_t(x',x_n+y_n)-G_t(x)]\bigr\Vert_{L^1_x(\mathbb{R}^n_+)}\\
       \leq&
        C\int_0^\infty(t+\tau)^{-\frac{3}{2}}\int_0^\infty|\int_0^{x_n}G_{\tau}^{(1)}(x_n-z_n)[G_t(z_n+y_n)-G_t(z_n)]\,dz_n|\,dx_n\,d\tau\\
        &+\Vert 2\nabla_{x'}'[G_t(x',x_n+y_n)-G_t(x)]\Vert_{L^1_x(\mathbb{R}^n_+)}\\
        \leq & C\int_0^\infty(t+\tau)^{-\frac{3}{2}}\int_0^\infty\int_{z_n}^\infty |G_\tau^{(1)}(x_n-z_n)[G_t(z_n+y_n)-G_t(z_n)]|\,dx_n\,dz_n\,d\tau\\
        &+2\int_{\mathbb{R}^{n-1}}|\nabla_{x'}'G_t(x')|\,dx'\int_{0}^\infty |G_t(x_n+y_n)-G_t(x_n)|\,dx_n\\
      \leq & C\int_0^\infty(t+\tau)^{-\frac{3}{2}}\,d\tau\int_{0}^\infty |G_t(z_n+y_n)-G_t(z_n)|\,dz_n\\
      &+Ct^{-\frac{1}{2}}\int_{0}^\infty |G_t(x_n+y_n)-G_t(x_n)|\,dx_n\\
    \leq & Ct^{-\frac{1}{2}}\int_{0}^\infty |G_t(z_n+y_n)-G_t(z_n)|\,dz_n.
    \end{aligned}
\end{equation}
This, together with   (\ref{3.17}), gives
\begin{equation}\label{sl2}
    \begin{aligned}
        &\lim_{t\rightarrow +\infty}\left\Vert t^{\frac{1}{2}}\int_{\mathbb{R}^n_+}[\partial_{x_n}\nabla_{x'}'M_n^*(x,0',y_n,t)-\partial_{x_n}\nabla_{x'}'M_n^*(x,0,t)]\cdot y'a_n(y)\,dy\right\Vert_{L^1_x(\mathbb{R}^n_+)}\\\leq&\lim_{t\rightarrow +\infty}\sum_{j=1}^{n-1}t^{\frac{1}{2}}\int_{\mathbb{R}^n_+}\Vert\partial_{x_n}\nabla_{x'}'M_n^*(x,0',y_n,t)-\partial_{x_n}\nabla_{x'}'M_n^*(x,0,t)\Vert_{L^1_x(\mathbb{R}^n_+)}|y'a_n(y)|\,dy\\\leq &C\lim_{t\rightarrow +\infty}\int_{\mathbb{R}^n_+}\int_0^\infty| G_t^{(1)}(z_n+y_n)-G_t^{(1)}(z_n)|\,dz_n |y'a_n(y)|\,dy\\=&0,
    \end{aligned}
\end{equation}
where the dominated convergence theorem has been used in the last equality.
Combining (\ref{3.13})--(\ref{yuxiang1}), (\ref{chafen2}), (\ref{sl1}) and (\ref{sl2}), yields that for $i=1,2,\dots,n$
\begin{equation}\label{v^2}\begin{aligned}
 \lim_{t\to +\infty}t^{\frac{1}{2}}\biggl\|v_i^{(2)}(x,t)+\sum_{j=1}^{n-1}L_{ij}(x,t)\int_{\mathbb{R}^n_+}G_t^{(1)}(y_n)y_ja_n(y)\,dy\\+\sum_{j=1}^{n-1}\partial_{x_n}\partial_{x_j}M_i^*(x,0,t)\int_{\mathbb{R}^n_+}y_ja_n(y)\,dy\biggr\|_{L^1_x(\mathbb{R}^n_+)}=0.
\end{aligned}\end{equation}
Clearly, (\ref{high order}) is a direct consequence (\ref{v^1}) and (\ref{v^2}).

\textbf{Step 5. Analysis of leading term.} The estimates (\ref{pointwise}) are given by Lemma \ref{lem:L_{ij}}. In the process of proving the necessary and sufficient condition, we  focus on the second term of the right-hand side in (\ref{3.14}) as follows\[\begin{aligned}&\sum_{j=1}^{n-1}4\partial_{x_j}\int_0^\infty\partial_{x_i}[G_{t+\tau}^{(n-1)}(x)G_\tau^{(1)}(x_n)]
\,d\tau\int_{\mathbb{R}^n_+}G_t^{(1)}(y_n)y_j a_n(y',y_n)\,dy'\,dy_n\\=& \sum_{j=1}^{n-1}L_{ij}(x,t)\int_{\mathbb{R}^n_+}G_t^{(1)}(y_n)y_j a_n(y',y_n)\,dy'\,dy_n.\end{aligned}\]

If (\ref{condition}) holds, one has $\boldsymbol{v}(\cdot,t)\in L^1(\mathbb{R}^n_+)$ obviously. On the other hand, assume that (\ref{condition}) fails. 
Suppose
\[
\int_{\mathbb{R}^n_+}G_t^{(1)}(y_n)\,y_i\,a_n(y)\,dy\neq 0
\quad\text{for some }i=1,2,\dots,n-1.
\]
Then  it follows from (\ref{pointwise}) that there exists a $M(t)>0$ large enough and $C(t),\widetilde{C}(t)>0$ such that for $x\in \Lambda_{M(t)}\cap \left(B(0,M(t))\right)^c$
\begin{equation}\label{arg}
\frac{C(t)}{|x|^n}\leq\left|\sum_{j=1}^{n-1}L_{ij}(x,t)\int_{\mathbb{R}^n_+}G_t^{(1)}(y_n)y_ja_n(y)\,dy\right|\leq \frac{\widetilde{C}(t)}{|x|^n}.
\end{equation}
It follows from (\ref{linear}) that $v_i(x,t)$ does not belong to $L^1(\mathbb{R}^n_+)$.
Hence the proof for Theorem \ref{thm:1.1} is completed.\end{proof}
\section{Long time behavior for solutions of the Navier--Stokes system in $L^1(\mathbb{R}^n_+)$ }\label{sec:4}
We now consider the nonlinear problem (\ref{1.1}). The existence of strong solution of (\ref{1.1}) is given by Lemma \ref{lem:2.1}. And if $\mathcal{R}(\cdot,t)\in L^1(\mathbb{R}^n_+)$, by (\ref{cond:decay}) and the same argument as \textbf{Step 5} in the proof of Theorem \ref{thm:1.1}, we have $\boldsymbol{u}(\cdot,t) \in L^1(\mathbb{R}^n_+)$ if and only if $\mathscr{A}_j[\boldsymbol{a},\boldsymbol{u}](t) = 0$ for  each $j = 1,2, \dots, n-1$. Hence, the rest of this section is devoted to the proof of Case (1)  in Theorem \ref{thm:1.2} which concerns long time behavior of the solutions of (\ref{1.1}).

   Let $g=\mathcal{N}f$ and $g=\mathcal{D}f$   denote the solution to the Neumann problem and the Dirichlet problem
   \begin{equation}
    \begin{cases}
        -\Delta g = f          & \text{in } \mathbb{R}^n_+, \\
        |g(x)| \to 0           & \text{as } |x| \to \infty, \\
        \partial_n g = 0       & \text{on } \partial\mathbb{R}^n_+
    \end{cases}\quad \text{and}\quad \begin{cases}
        -\Delta g = f          & \text{in } \mathbb{R}^n_+, \\
        |g(x)| \to 0           & \text{as } |x| \to \infty, \\
         g = 0       & \text{on } \partial\mathbb{R}^n_+,
    \end{cases}
\end{equation}
respectively. Then the straightforward computations show\begin{equation}
        \mathcal{N}f=\int_{\mathbb{R}^n_+}[E(x-y)+E(x-y^*)]f(y)\,dy
    \end{equation}
and\begin{equation}
    \mathcal{D}f=\int_{\mathbb{R}^n_+}[E(x-y)-E(x-y^*)]f(y)\,dy.
\end{equation}    
Moreover, for any $\boldsymbol{u},\boldsymbol{v}\in L^2_\sigma(\mathbb{R}^n_+)\cap H^1_0(\mathbb{R}^n_+)$, one has (see \cite{2012H})
\begin{equation}\label{4.44}
\mathbb{P}(\boldsymbol{u}\cdot\nabla \boldsymbol{v})=\boldsymbol{u}\cdot\nabla \boldsymbol{v}+\sum_{i,j=1}^{n}\nabla\mathcal{N}\partial_i\partial_j(u_iv_j).
\end{equation}
    It follows from Duhamel formula$$\boldsymbol{u}(x,t)=e^{-t\mathbb{A}}\boldsymbol{a}-\int_0^t e^{-(t-s)\mathbb{A}}\mathbb{P}\left(\boldsymbol{u}\cdot\nabla \boldsymbol{u}(s)\right)\,ds,$$
    that one can select the part not in $L^1(\mathbb{R}^n_+)$ from the second term in the same way of Theorem \ref{1.1}. In this case, $a_n$ is replaced by $\left(\mathbb{P}(\boldsymbol{u}\cdot \nabla \boldsymbol{u})\right)_n$ so the coefficients of
$L^1(\mathbb{R}^n_+)$  characterization can be represented by \begin{equation}\label{4.16}\int_0^t\int_{\mathbb{R}^n_+}G_{t-s}^{(1)}(y_n)y_k\left(\mathbb{P}(\boldsymbol{u}\cdot\nabla \boldsymbol{u}(y,s))\right)_n\,dy\,ds=\int_0^t\int_{\mathbb{R}^n_+}G_{t-s}^{(1)}(y_n)u_nu_k(y,s)\,dy\,ds.\end{equation} More precisely, for $k<n$, it follows from (\ref{4.44}) that one has
    \begin{equation}
        \begin{aligned}
            &\int_0^t\int_{\mathbb{R}^n_+}G_{t-s}^{(1)}(y_n)y_k\left(\mathbb{P}(\boldsymbol{u}\cdot\nabla \boldsymbol{u}(y,s))\right)_n\,dy\,ds\\=&\int_0^t\int_{\mathbb{R}^n_+}G_{t-s}^{(1)}(y_n)y_k\left(\boldsymbol{u}\cdot\nabla u_n+\sum_{i,j=1}^n\partial_{y_n}\mathcal{N}\partial_i\partial_j(u_iu_j)\right)\,dy\,ds.
        \end{aligned}
    \end{equation} Using $\nabla\cdot \boldsymbol{u}=0$ yields
    \begin{equation}\label{4.18}
        \begin{aligned}
           &\int_{\mathbb{R}^{n-1}}y_k\boldsymbol{u}\cdot\nabla u_n\,dy'\\=&-\int_{\mathbb{R}^{n-1}}y_k\sum_{j\neq k,n}(\partial_{j}u_j)u_n\,dy'+\int_{\mathbb{R}^n_+}y_ku_k\partial_{k}u_n\,dy'+\int_{\mathbb{R}^{n-1}}y_ku_n\partial_{n}u_n\,dy'\\=&-\int_{\mathbb{R}^n_+}u_ku_n\,dy'+\int_{\mathbb{R}^{n-1}}y_k\partial_{y_n}(u_nu_n)\,dy'.
        \end{aligned}
    \end{equation}
    For $i,j<n$; $i=n,j\neq k,n$ or $j=n, i\neq k,n$, one has
    \begin{equation}\label{4.19}
        \int_{\mathbb{R}^{n-1}}y_k\partial_{y_n}\mathcal{N}\partial_i\partial_j(u_iu_j)\,dy'=0.
    \end{equation}
    For $i=k,j=n$ or $i=n,j=k,$ it holds that
    \begin{equation}\label{4.20}
        \begin{aligned}
           \int_{\mathbb{R}^{n-1}}y_k\partial_{y_n}\mathcal{N}\partial_k\partial_n(u_ku_n)\,dy'
&=\int_{\mathbb{R}^{n-1}}y_k\partial_{y_n}^2\mathcal{D}\partial_k(u_ku_n)\,dy'\\&=-\int_{\mathbb{R}^{n-1}}y_k\sum_{j=1}^{n-1}\partial_{y_j}^2\mathcal{D}\partial_k(u_ku_n)\,dy'-\int_{\mathbb{R}^{n-1}}y_k\partial_k(u_ku_n)\,dy'\\&=\int_{\mathbb{R}^{n-1}}u_nu_k\,dy'.        \end{aligned}
    \end{equation}
    For $i=j=n,$ one has
    \begin{equation}\label{4.21}
        \begin{aligned}
            \int_{\mathbb{R}^{n-1}}y_k\partial_{y_n}^2\mathcal{N}\partial_n(u_nu_n)\,dy'&=-\int_{\mathbb{R}^{n-1}}y_k\sum_{j=1}^{n-1}\partial_{y_j}^2\mathcal{N}\partial_n(u_nu_n)\,dy'-\int_{\mathbb{R}^{n-1}}y_k\partial_{y_n}(u_nu_n)\,dy'\\&=-\int_{\mathbb{R}^{n-1}}y_k\partial_{y_n}(u_nu_n)\,dy'.
        \end{aligned}
    \end{equation}
    Combining (\ref{4.18})--(\ref{4.21}) gives (\ref{4.16}).
    But we cannot derive the same estimates of remainder terms directly  since $|y'|\mathbb{P}(\boldsymbol{u}\cdot\nabla \boldsymbol{u})$ might not be in $L^1(\mathbb{R}^n_+)$.

    The following lemma can be regarded as an improvement of \cite[Theorem 1.4]{2018H}, in the sense that the decay rate is improved to $-\frac{1}{2}$. 
    \begin{Lemma}\label{lem:4.1}
        Suppose the initial data $\boldsymbol{a}\in L^2_\sigma(\mathbb{R}^n_+) \cap L^1(\mathbb{R}^n_+) \cap L^n(\mathbb{R}^n_+)$  satisfies the same condition in Case (1) of Theorem \ref{thm:1.2}. Then the strong solution $\boldsymbol{u}$ of (\ref{1.1}) given by Lemma \ref{lem:2.1} satisfies, for any $t>0$
        
        \begin{equation}\label{profile}\begin{aligned}
            \left\| \boldsymbol{u}-e^{-t\mathbb{A}}\boldsymbol{a}-4\nabla_{x}\sum_{j=1}^{n-1}\partial_{x_j}\int_0^\infty G_{t+\tau}^{(n-1)}(x')G_\tau^{(1)}(x_n)\,d\tau\int_0^t\int_{\mathbb{R}^n_+}G_{t-s}^{(1)}(y_n)u_nu_j(y,s)\,dy\,ds\right\|_{L^1(\mathbb{R}^n_+)}\leq C t^{-\frac{1}{2}}.\end{aligned}\end{equation}\end{Lemma}
            \begin{Remark}
              In the proof of  Lemma \ref{lem:4.1}, it can be observed that for the characterization extraction of $L^1$, a key point is to obtain the weighted $L^2$ estimate of the solution $\||x'|^{\frac{\alpha}{2}}\boldsymbol{u}\|_{L^2(\mathbb{R}^n_+)}$ for some $\alpha\in (0,1)$. However, due to the lack of information about the pressure near the boundary, unlike  the solution of the Cauchy problem \cite{2001HX}, the  presence of boundary requires additional regularity on the initial data to obtain the weighted energy estimates for weak solutions in the half space (cf.\cite{2009HW}).  
            \end{Remark}
       \begin{proof}
           The proof is carried out in three steps.

   \textbf{Step 1. Analysis of nonlinear part.} Note that $$\begin{aligned}\boldsymbol{u}(x,t)&=e^{-t\mathbb{A}}\boldsymbol{a}-\int_0^t e^{-(t-s)\mathbb{A}}\mathbb{P}(\boldsymbol{u}\cdot\nabla \boldsymbol{u}(s))\,ds\\&=e^{-t\mathbb{A}}\boldsymbol{a}-\int_0^t e^{-(t-s)\mathbb{A}}\mathbb{P}\nabla\cdot (\boldsymbol{u}\otimes \boldsymbol{u})\,ds.\end{aligned}$$
   
  Let $\mathscr{F}=\boldsymbol{u}\otimes \boldsymbol{u}$ in (\ref{4.40}). For $i=1,2,\dots,n$, $x\in \mathbb{R}^n_+$ and $t>0,$ one has
   \begin{equation}\label{4.45}
    u_i(x,t)-[e^{-t\mathbb{A}}\boldsymbol{a}]_i
    := -\sum_{j=1}^{9} I_j(x,t),
\end{equation}where\begin{equation}
    I_1 =
    -\sum_{j\neq n}
    \sum_{l=1}^{n}
    \int_0^t \int_{\mathbb{R}^n_+}
    \partial_{y_l}\mathcal{G}_{ij}(x,y,t-s)\,
    u_l \, u_j(y,s) \, dy \, ds
\end{equation}and
   \begin{equation}\small
        \begin{aligned}
            &\sum_{j=2}^{9} I_j=\sum_{j\neq n}\int_0^t\int_{\mathbb{R}^n_+}\partial_{y_j}\mathcal{G}_{ij}(x,y,t-s)u_nu_n(y,s)\,dy\,ds+\sum_{l\neq n}\int_0^t\int_{\mathbb{R}^n_+}\partial_{y_l}\mathcal{G}_{in}(x,y,t-s)u_nu_l(y,s)\,dy\,ds\\&-\sum_{k,j,l\neq n}\int_0^t\int_{\mathbb{R}^n_+}\partial_{x_j}\partial_{x_l}K_{ijk}^+(x,y,t-s)u_lu_k(y,s)\,dy\,ds-\sum_{j,l\neq n}\int_{\mathbb{R}^n_+} \partial_{x_j}\partial_{x_l}K_{ijn}^-(x,y,t-s)(u_lu_n+u_nu_l)(y,s)\,dy\\&+\sum_{j,m\neq n}\int_0^t\int_{\mathbb{R}^n_+} \partial^2_{x_m}K_{ijj}^+(x,y,t-s)u_nu_n(y,s)\,dy\,ds-\sum_{k,l\neq n}\int_0^t\int_{\mathbb{R}^n_+} \partial_{x_k}\partial_{x_l}K_{inn}^-(x,y,t-s)u_lu_k(y,s)\,dy\,ds\\&+
    \sum_{l,m\neq n}\int_0^t\int_{\mathbb{R}^n_+} \partial^2_{x_m}K_{inl}^+(x,y,t-s)(u_lu_n+u_nu_l)(y,s)\,dy\,ds+\sum_{m\neq n}\int_0^t\int_{\mathbb{R}^n_+}\partial_{x_m}^2K^-_{inn}(x,y,t-s)u_nu_n(y,s)\,dy\,ds.
        \end{aligned}
        \end{equation}

     \textbf{Step 2. Estimates for $I_2$--$I_9$}. Note that $M_{in}(x,y,t)=0$ for $i=1,2,\dots,n,$ then
        \begin{equation}\begin{aligned}
           & \sum_{l\neq n}\int_0^t\int_{\mathbb{R}^n_+}\partial_{y_l}\mathcal{G}_{in}(x,y,t-s)u_lu_n(y,s)\,dy\,ds\\=&\sum_{l\neq n}\int_0^t\int_{\mathbb{R}^n_+}\partial_{y_l}[G_{t-s}(x-y)-G_{t-s}(x-y^*)]u_lu_n(y,s)\,dy\,ds.
        \end{aligned}\end{equation}
       Note that  for $|k|=0,1,2,\dots$\begin{equation}\label{heat}
    |\partial_{x}^kG_t(x)|\leq C\frac{t^{\frac{\theta-n-|k|}{2}}}{(|x|+\sqrt{t})^\theta} \quad \text{for}\ \theta\geq 0.
\end{equation} It follows from (\ref{2.9}), (\ref{6.47}), (\ref{heat}) and (\ref{two}) that, for  $t>0$, one has\begin{equation}\label{qiuhe}
    \begin{aligned}
        \left\|\sum_{j=2}^9I_j\right\|_{L^1(\mathbb{R}^n_+)}&\leq C\int_{0}^{t}\int_{\mathbb{R}^{n}_{+}}\|(|x-y|+\sqrt{t-s})^{-(n+1)}\|_{L_x^1(\mathbb{R}^n_+)}|\boldsymbol{u}(y,s)|^2\,dy\,ds\\& \leq C\int_{0}^{t}\|(|x-y|+\sqrt{t-s})^{-(n+1)}\|_{L_x^1(\mathbb{R}^n_+)}\|\boldsymbol{u}\|^2_{L^2(\mathbb{R}^n_+)}\,ds\\&\leq C\int_0^t(t-s)^{-\frac{1}{2}}(1+s)^{-\frac{n+2}{2}}\,ds\\&\leq Ct^{-\frac{1}{2}}\int_0^\frac{t}{2}(1+s)^{-\frac{n+2}{2}}\,ds+C\int_\frac{t}{2}^t(t-s)^{-\frac{1}{2}}(1+s)^{-\frac{n+2}{2}}\,ds\\&\leq Ct^{-\frac{1}{2}}.
    \end{aligned}
\end{equation}

\textbf{Step 3. Estimates for $I_1$}. Note that\begin{equation}\label{I1}
    \begin{aligned}
        I_1&=-\sum_{l=1}^{n}\sum_{j=1}^{n-1}\delta_{ij}\int_0^t\int_{\mathbb{R}^n_+}\partial_{y_l}[G_{t-s}(x-y)-G_{t-s}(x-y^*)]u_lu_j(y,s)\,dy\,ds\\&\quad-\sum_{l=1}^{n-1}\sum_{j=1}^{n-1}\int_0^t\int_{\mathbb{R}^n_+}\partial_{y_l}M_{ij}(x,y,t-s)u_lu_j(y,s)\,dy\,ds\\&\quad-\sum_{j=1}^{n-1}\int_0^t\int_{\mathbb{R}^n_+}\partial_{y_n}[M_{ij}(x,y,t-s)-M_{ij}(x,0',y_n,t-s)]u_nu_j(y,s)\,dy\,ds\\&\quad-\sum_{j=1}^{n-1}\int_0^t\int_{\mathbb{R}^n_+}\partial_{y_n}M_{ij}(x,0',y_n,t-s)u_nu_j(y,s)\,dy\,ds\\&=I_{11}+I_{12}+I_{13}+I_{14}.
    \end{aligned}
\end{equation}
It follows from (\ref{2.9}),
 (\ref{heat}), (\ref{2.9}) and (\ref{two}) that for  $t>0$, we have 
\begin{equation}\label{4.60}
    \begin{aligned}
        \|I_{11}+I_{12}\|_{L^1(\mathbb{R}^n_+)}&\leq C\int_{0}^{t}\int_{\mathbb{R}^{n}_{+}}\|(|x-y|+\sqrt{t-s})^{-(n+1)}\|_{L^1_x(\mathbb{R}^n_+)}|\boldsymbol{u}(y,s)|^2\,dy\,ds\\& \leq C\int_{0}^{t}\|(|x-y|+\sqrt{t-s})^{-(n+1)}\|_{L_x^1(\mathbb{R}^n_+)}\|\boldsymbol{u}\|^2_{L^2(\mathbb{R}^n_+)}\,ds\\&\leq C\int_0^t(t-s)^{-\frac{1}{2}}(1+s)^{-\frac{n+2}{2}}\,ds\\&\leq Ct^{-\frac{1}{2}}\int_0^\frac{t}{2}(1+s)^{-\frac{n+2}{2}}\,ds+C\int_\frac{t}{2}^t(t-s)^{-\frac{1}{2}}(1+s)^{-\frac{n+2}{2}}\,ds\\&\leq Ct^{-\frac{1}{2}}.
    \end{aligned}
\end{equation}
Using Hölder inequality yields for $y\in\mathbb{R}^n_+$, $t>s>0$ and $\alpha\in (0,1)$\begin{equation}
    \begin{aligned}
        &\left\|\partial_{y_n}\left(M_{ij}(x,y,t-s)-M_{ij}(x,0',y_n,t-s)\right)\right\|_{L^1_{x}(\mathbb{R}^n_+)}\\=&\int_{\mathbb{R}^n_+}\left|\partial_{y_n}\int_0^1y'\cdot\nabla_{\theta y'}' M_{ij}(x,\theta y',y_n,t-s)\,d\theta\right|^{\alpha}\\&\quad\times\left|\partial_{y_n}M_{ij}(x,y,t-s)-\partial_{y_n}M_{ij}(x,0',y_n,t-s)\right|^{1-\alpha}\,dx\\\leq&
       |y'|^\alpha \int_0^\infty\int_{\mathbb{R}^{n-1}}\left(\int_0^1(t-s)^{-\frac{1}{2}}(|x'-\theta y'|+x_n+y_n+\sqrt{t-s})^{-n-1}\,d\theta\right)^\alpha\\&\quad\times\left((t-s)^{-\frac{1}{2}}\left((|x'- y'|+x_n+y_n+\sqrt{t-s})^{-n}+(|x'|+x_n+y_n+\sqrt{t-s})^{-n}\right)\right)^{1-\alpha}\,dx'\,dx_n\\\leq &C|y'|^\alpha(t-s)^{-\frac{1}{2}}\int_0^\infty\left(\int_{\mathbb{R}^{n-1}}(|x'-\theta y'|+x_n+y_n+\sqrt{t-s})^{-n-1}\,dx\,d\theta\right)^\alpha\\&\quad\times\left(\int_{\mathbb{R}^{n-1}}\left((|x'- y'|+x_n+y_n+\sqrt{t-s})^{-n}+(|x'|+x_n+y_n+\sqrt{t-s})^{-n}\right)\,dx'\right)^{1-\alpha}\,dx_n\\\leq&C|y'|^\alpha(t-s)^{-\frac{1}{2}}\int_0^\infty(x_n+y_n+\sqrt{t-s})^{-1-\alpha}\,dx_n\\\leq& C|y'|^\alpha(t-s)^{-\frac{1+\alpha}{2}}.
    \end{aligned}
\end{equation}
This, together with (\ref{three}) implies that for $\alpha\in (0,1)$ and $t>0$, one has 
\begin{equation}\label{I13}
    \begin{aligned}
        &\|I_{13}\|_{L^1(\mathbb{R}^n_+)}\\=&\sum_{j=1}^{n-1}\int_0^t\int_{\mathbb{R}^n_+}\left\|\partial_{y_n}[M_{ij}(x,y,t-s)-M_{ij}(x,0',y_n,t-s)]\right\|_{L^1_x(\mathbb{R}^n_+)}(u_nu_j)(y,s)\,dy\,ds\\\leq& C\int_0^t(t-s)^{-\frac{1+\alpha}{2}}\||y'|^{\frac{\alpha}{2}}\boldsymbol{u}\|^2_{L^2(\mathbb{R}^n_+)}\,ds\\\leq& C\int_0^t(t-s)^{-\frac{1+\alpha}{2}}(1+s)^{-\frac{n-\alpha}{2}}\,ds\\\leq& Ct^{-\frac{1+\alpha}{2}}\int_0^\frac{t}{2}(1+s)^{-\frac{n-\alpha}{2}}\,ds+C\int_\frac{t}{2}^t(t-s)^{-\frac{1+\alpha}{2}}(1+s)^{-\frac{n-\alpha}{2}}\,ds\\\leq& 
        C t^{-\frac{1}{2}}(1+t^{-\frac{\alpha}{2}}).\end{aligned}
\end{equation}Additionally, from the fourth line of (\ref{I13}) we obtain
\[
\|I_{13}\|_{L^1(\mathbb{R}^n_+)}\leq C\int_0^t(t-s)^{-\frac{1+\alpha}{2}}(1+s)^{-\frac{n-\alpha}{2}}\,ds \leq C t^{\frac{1-\alpha}{2}}\quad\text{for }\alpha\in (0,1)\text{ and }t>0.
\]
Combining this with (\ref{I13}) yields
\begin{equation}\label{i13}
\|I_{13}\|_{L^1(\mathbb{R}^n_+)} \leq C t^{-\frac{1}{2}} \quad\text{for } t>0.
\end{equation}

Applying Lemma \ref{lem:2.4} gives
\begin{equation}\label{I14}
    \begin{aligned}
        I_{14}&=-\sum_{j=1}^{n-1}\int_0^t\int_{\mathbb{R}^n_+}\partial_{x_n}M_{ij}(x,0',y_n,t-s)u_nu_j(y,s)\,dy\,ds\\&\quad-\sum_{j=1}^{n-1}\int_0^t\int_{\mathbb{R}^n_+}[L_{ij}(x,t-s)-L_{ij}(x,t)]G^{(1)}_{t-s}(y_n)u_nu_j(y,s)\,dy\,ds\\&\quad-\sum_{j=1}^{n-1}L_{ij}(x,t)\int_0^t\int_{\mathbb{R}^n_+}G^{(1)}_{t-s}(y_n)u_nu_j(y,s)\,dy\,ds\\&=I_{141}+I_{142}+I_{143}.
    \end{aligned}
\end{equation}
Using (\ref{2.9}) and (\ref{two}) yields that for $t>0$ \begin{equation}\label{i141}\begin{aligned}
     &\|I_{141}\|_{L^1(\mathbb{R}^n_+)}\\\leq& C\int_0^t\int_{\mathbb{R}^n_+}\left\|\partial_{x_n}M_{ij}(x,0',y_n,t-s)\right\|_{L^1_x(\mathbb{R}^n_+)}|u_nu_j(y,s)|\,dy\,ds\\\leq& C\int_0^t\left\|(|x'|+x_n+y_n+\sqrt{t-s})^{-n}(x_n+\sqrt{t-s})^{-1}\right\|_{L^1_x(\mathbb{R}^n_+)}\|\boldsymbol{u}\|^2_{L^2(\mathbb{R}^n_+)}\,ds\\\leq& C\int_0^t(t-s)^{-\frac{1}{2}}(1+s)^{-\frac{n+2}{2}}\,ds\\\leq& Ct^{-\frac{1}{2}}\int_0^\frac{t}{2}(1+s)^{-\frac{n+2}{2}}\,ds+C\int_\frac{t}{2}^t(t-s)^{-\frac{1}{2}}(1+s)^{-\frac{n+2}{2}}\,ds \\\leq& Ct^{-\frac{1}{2}}.
 \end{aligned}\end{equation} Recall the definition (\ref{L_{ij}}) of $L_{ij}(x,t)$, for  $j<n$ and $t>0$, we have
 \begin{equation}
 \begin{aligned}
     &\| \partial_tL_{ij}(x,t)\|_{L^1(\mathbb{R}^n_+)}\\\leq& \begin{cases}
        4\int_0^\infty \|\partial_t\partial_{x_i}\partial_{x_j} G_{t+\tau}^{(n-1)}(x')\|_{L^1_{x'}(\mathbb{R}^{n-1})}\|G_{\tau}^{(1)}(x_n)\|_{L^1_{x_n}((0,\infty))}\,d\tau & \text{if } i<n, \\
        4\int_0^\infty \|\partial_t\partial_{x_j} G_{t+\tau}^{(n-1)}(x')\|_{L^1_{x'}(\mathbb{R}^{n-1})}\|\partial_{x_n}G_{\tau}^{(1)}(x_n)\|_{L^1_{x_n}((0,\infty))}\,d\tau                  & \text{if } i=n,
    \end{cases}\\\leq& \begin{cases}
       C\int_0^\infty(t+\tau)^{-2}\,d\tau & \text{if } i<n, \\
        C\int_0^\infty (t+\tau)^{-\frac{3}{2}}\tau^{-\frac{1}{2}}\,d\tau                  & \text{if } i=n,
    \end{cases}\\\leq& Ct^{-1}.
 \end{aligned}
 \end{equation}
From this, together with (\ref{two}), one has
\begin{equation}\label{I_{142}}\begin{aligned}&\|I_{142}\|_{L^1(\mathbb{R}^n_+)}\\\leq&\int_0^t\int_{\mathbb{R}^n_+}\|L_{ij}(x,t)-L_{ij}(x,t-s)\|_{L^1_x(\mathbb{R}^n_+)}G^{(1)}_{t-s}(y_n)|u_nu_j|\,dy\,ds\\\leq& C\int_0^t\int_{\mathbb{R}^n_+} \left(\int_0^1s\|\partial_tL_{ij}(x,t-\theta s)\|_{L^1_x(\mathbb{R}^n_+)}\,d\theta\right)(t-s)^{-\frac{1}{2}}e^{-\frac{y_n^2}{4(t-s)}}|u_nu_j|\,dy\,ds\\\leq &C\int_0^t\left(\int_0^1 (t-\theta s)^{-1}s \,d\theta\right)(t-s)^{-\frac{1}{2}}\| \boldsymbol{u}\|^2_{L^2(\mathbb{R}^n_+)}\,ds\\\leq& C\int_0^t\log \frac{t}{t-s}(t-s)^{-\frac{1}{2}}(1+s)^{-\frac{n+2}{2}}\,ds\\ \leq& C\int_0^\frac{t}{2}\log \frac{t}{t-s}(t-s)^{-\frac{1}{2}}(1+s)^{-\frac{n+2}{2}}\,ds+C\int_\frac{t}{2}^t\log \frac{t}{t-s}(t-s)^{-\frac{1}{2}}(1+s)^{-\frac{n+2}{2}}\,ds\\\leq& Ct^{-\frac{1}{2}}\int_0^\frac{t}{2}(1+s)^{-\frac{n+2}{2}}\,ds+C\left(1+\frac{t}{2}\right)^{-\frac{n+2}{2}}t^{\frac{1}{2}}\int_{\frac{1}{2}}^1\log\frac{1}{1-s}(1-s)^{-\frac{1}{2}}\,ds \\\leq& Ct^{-\frac{1}{2}}.\end{aligned}.\end{equation} Combining (\ref{qiuhe}),(\ref{4.60}),(\ref{i13}),(\ref{i141}) and (\ref{I_{142}})  finishes the proof of the lemma.
       \end{proof} 

    \begin{proof}[Proof of Case (1) in Theorem \ref{thm:1.2}]
    First, the estimate (\ref{1.10}) can be derived directly from (\ref{linear}) in Theorem \ref{thm:1.1} and Lemma \ref{lem:4.1}.

     Next we  establish the first sufficient condition on initial data for $\boldsymbol{u}(\cdot,t)\notin L^1$. For $n\geq 3$, it follows from (\ref{one}) in Lemma \ref{lem:decay} that there exists a constant $C > 0$, depending only on $n$,  such that the solution $\boldsymbol{u}$ of (\ref{1.1}) satisfies
\begin{equation}
\Vert \boldsymbol{u}(t)\Vert_{L^2(\mathbb{R}^n_+)} \leq C\left(K(\boldsymbol{a})\right)^{\frac{1}{2}}(1+t)^{-\frac{n}{4}}\quad\text{for }t>0,
\end{equation}
where $K(\boldsymbol{a})$ is defined in (\ref{qingk}).

       Therefore, for $j<n$, one has\begin{equation}\label{eng}
            \begin{aligned}
                &\left|\int_0^t\int_{\mathbb{R}^n_+}G_{t-s}^{(1)}(y_n)u_nu_j(y,s)\,dy\,ds\right|\\\leq &C \int_0^t(t-s)^{-\frac{1}{2}}\Vert \boldsymbol{u}(\cdot,s)\Vert_{L^2(\mathbb{R}^n_+)}^2\,ds\\\leq &CK(\boldsymbol{a})\int_0^t(t-s)^{-\frac{1}{2}}(1+s)^{-\frac{n}{2}}\,ds\\\leq &CK(\boldsymbol{a})\left(\int_0^{\frac{t}{2}}(t-s)^{-\frac{1}{2}}(1+s)^{-\frac{n}{2}}\,ds+\int^t_{\frac{t}{2}}(t-s)^{-\frac{1}{2}}(1+s)^{-\frac{n}{2}}\,ds\right)\\\leq &C_1K(\boldsymbol{a})t^{-\frac{1}{2}},
            \end{aligned}
        \end{equation}
        where we used $n\geq 3$ in the fourth line.

        For $n=2$, it follows from (\ref{two}) in Lemma \ref{lem:decay} that\begin{equation}
\Vert \boldsymbol{u}(t)\Vert_{L^2(\mathbb{R}^n_+)} \leq C\left(K(\boldsymbol{a})\right)^{\frac{1}{2}}(1+t)^{-\frac{n+2}{4}}\quad\text{for }t>0.
\end{equation} One has
\begin{equation}
    \left|\int_0^t\int_{\mathbb{R}^n_+}G_{t-s}^{(1)}(y_n)u_nu_j(y,s)\,dy\,ds\right|\leq C_1K(\boldsymbol{a})t^{-\frac{1}{2}}
\end{equation}by the same argument as (\ref{eng}).

        On the other hand, there exists a $T_1>0$ large enough such that for $t>T_1$\begin{equation}
            \begin{aligned}
               \left|\int_{\mathbb{R}^n_+}G_t^{(1)}(y_n)y_j a_n(y)\,dy\right|=C_2t^{-\frac{1}{2}}\left|\int_{\mathbb{R}^n_+}e^{-\frac{y_n^2}{4t}}y_j a_n(y)\,dy\right|\geq \frac{C_2}{2}t^{-\frac{1}{2}}\left|\int_{\mathbb{R}^n_+}y_j a_n(y)\,dy\right|.
            \end{aligned}
        \end{equation}
        Here the dominated convergence theorem has been used.

      Comparing the linear and nonlinear parts and choosing $\epsilon= \frac{C_2}{2C_1}$ in (\ref{epsilon}), yield
\[
\left|\mathscr{A}_j[\boldsymbol{a},\boldsymbol{u}](t)\right| \geq \left| \int_{\mathbb{R}^n_+} G_t^{(1)}(y_n) y_j a_n(y) \,dy \right|
- \left| \int_0^t \int_{\mathbb{R}^n_+} G_{t-s}^{(1)}(y_n) u_n u_j(y,s) \,dy \,ds \right| > 0\quad \text{for } t>T_1.
\]
This finishes the proof of Case (1) in Theorem \ref{thm:1.2}.\end{proof}\section{Pointwise lower bound estimates for solutions of the Navier--Stokes system}\label{sec:5}
        In this section, we prove Case (2) of Theorem \ref{thm:1.2}. The next lemma shows that the same  estimate as (\ref{1.10}) holds for strong solution with pointwise decay. Here we can omit the condition on $\nabla \boldsymbol{a}$ because the weighted norm $\||x|^{\frac{\alpha}{2}}\boldsymbol{u}\|_{L^2(\mathbb{R}^n_+)}$ for $\alpha\in (0,1)$ of strong mild solution is already controlled by the pointwise assumptions on the initial data.
\begin{Lemma}\label{lem:5.1}
    Assume  the initial data $\boldsymbol{a}\in L^1(\mathbb{R}^n_+)\cap L^n(\mathbb{R}^n_+)$ satisfies (\ref{origin}) and the same condition in Case (2) of Theorem \ref{thm:1.2}, then the unique strong solution $\boldsymbol{u}\in L^\infty(\mathbb{R}^n_+\times (0,T_2))$ of (\ref{1.1}) given by Lemma \ref{lem:2.3} satisfies for  $t\in (0,T_2)$
    \begin{equation}
\begin{aligned}
    \biggl\lVert u_i(x,t) - \sum_{j=1}^{n-1} L_{ij}(x,t) \mathscr{A}_j[\boldsymbol{a},\boldsymbol{u}](t) \biggr\rVert_{L^1(\mathbb{R}^n_+)} 
    \leq  C t^{-\frac{1}{2}} (1 + t^{-\frac{1}{2}}).
\end{aligned}
\end{equation}
\end{Lemma}
\begin{proof}In the proof of Lemma \ref{lem:4.1}, we only used the following properties
\[
\|\boldsymbol{u}(t)\|_{L^2(\mathbb{R}^n_+)} \leq C(1+t)^{-\frac{n+2}{4}} \quad\text{and}\quad
\||x|^{\frac{\alpha}{2}} \boldsymbol{u}(t)\|_{L^2(\mathbb{R}^n_+)} \leq C(1+t)^{-\frac{n}{4}+\frac{\alpha}{4}}\quad\text{for } \alpha\in(0,1),  
\]
which now follow from (\ref{two}) and Lemma \ref{lem:2.3} respectively.
Hence (\ref{profile}) holds for $t\in (0,T_2)$ by the same argument in Lemma \ref{lem:4.1}. Combining this with (\ref{linear}) in Theorem \ref{thm:1.1} yields Lemma \ref{lem:5.1}.
\end{proof}Next we prove  Case (2) in Theorem \ref{thm:1.2}.\begin{proof}[Proof of Case (2) in Theorem \ref{thm:1.2}]   First, the estimate (\ref{1.10}) holds for $t\in (0,T_2)$ by Lemma \ref{lem:5.1} directly. Then we prove the spatial lower bound (\ref{lower}).

\textbf{Step 1. Nonlinear part $\int_0^t e^{-(t-s)\mathbb{A}}\mathbb{P}\nabla\cdot (\boldsymbol{u}\otimes \boldsymbol{u})\,ds$.} It follows from (\ref{cond:decay}) and Lemma \ref{lem:2.3} that$$
        |\boldsymbol{u}(x,t)| \leq \frac{C}{(1+|x|+\sqrt{t})^{n}} \quad\text{for } t\in(0,T_2).
    $$
From  (\ref{6.47}), (\ref{heat}) and Lemma \ref{lem:2.5}, using the definition of $I_m$ for $m=1,2,\dots,9$, $I_{11}$ and $I_{12}$ in (\ref{4.45}) and (\ref{I1}), yields for $t\in (0,T_2)$    
\begin{equation}\label{4.58}
    \begin{aligned}
        &\left|\sum_{m=2}^9I_m+I_{11}+I_{12}\right|\leq C\int_{0}^{t}\int_{\mathbb{R}^{n}_{+}}(|x-y|+\sqrt{t-s})^{-(n+1)}(1+|y|+\sqrt{s})^{-2n}\,dy\,ds\\\leq&C\int_{0}^{t}\int_{\mathbb{R}^{n}_{+}}(|x-y|+\sqrt{t-s})^{-(n+1)}(1+|y|+\sqrt{s})^{-(n+1)}\,dy\,ds\leq  C(1+|x|+\sqrt{t})^{-(n+1)}t^{\frac{1}{2}}.
    \end{aligned}
\end{equation}
 Recall the definition (\ref{domain}) of $\Lambda_{\sigma}$, $I_{13}$ in (\ref{I1}) and $I_{141}$ in (\ref{I14}), it follows from (\ref{2.9}) that for $x\in \Lambda_{\sigma}$ with $\sigma>1$ and $t\in (0,T_2)$, one has
\begin{equation}\label{4.61}
    \begin{aligned}
        |I_{13}|&=\sum_{j=1}^{n-1}\int_0^t\int_{\mathbb{R}^n_+}\partial_{y_n}[M_{ij}(x,y,t-s)-M_{ij}(x,0',y_n,t-s)](u_nu_j)(y,s)\,dy\,ds\\&\leq C\int_0^t\int_{\mathbb{R}^n_+}\left|\int_0^1\nabla_{\eta y'}' M_{ij}(x,\eta y',y_n,t-s)d\eta\right| |y'u_nu_j(y,s)|\,dy\,ds\\&\leq C\frac{1}{x_n^{n+1}}\int_0^t\int_{\mathbb{R}^n_+} |y'u_nu_j(y,s)|\,dy\,ds\\&\leq C\frac{t(1+\sqrt{t})^{-(n-1)}}{|x|^{n+1}}  \end{aligned}
\end{equation}
and
 \begin{equation}\label{4.63}
     |I_{141}|\leq C\int_0^t\int_{\mathbb{R}^n_+}(|x'|+x_n+y_n+\sqrt{t-s})^{-n}(x_n+\sqrt{t-s})^{-1}|u_nu_j(y,s)|\,dy\,ds\leq C\frac{t(1+\sqrt{t})^{-n}}{|x|^{n+1}}.
 \end{equation}
 Note that for $x\in\Lambda_{\sigma}$ with $\sigma>1$, $t>0$, $i\neq n,j<n$ and$\ j\neq i$,
 $$\begin{aligned}
     |\partial_t{L_{ij}(x,t)}|&=4\left|\int_0^\infty\partial_t\partial_{x_i}\partial_{x_j}[G_{\tau+t}^{(n-1)}(x')G_\tau^{(1)}(x_n)]\,d\tau\right|\\&\leq C\int_0^\infty\frac{|x_ix_j||x'|^2}{|x|^{n+8}}\frac{|x|^{n+8}}{\tau^{\frac{n}{2}+4}}e^{-(\frac{x_n}{|x|})^2\frac{|x|^2}{4\tau}}\,d\tau\\&\leq C\frac{|x_ix_j||x'|^2}{|x|^{n+6}}\int_0^\infty\lambda^{n+8}e^{-\frac{\sigma^2}{1+\sigma^2}\frac{\lambda^2}{4}}\,d\lambda\\&\leq C\frac{1}{|x|^{n+2}}.     
 \end{aligned}$$

 For $i=n,j<n$ or $i=j<n$, one can  get the above bound similarly.
 Therefore, for $i=1,2,\dots,n,j<n,t\in (0,T_2)$ and $x\in\Lambda_{\sigma}$ with $\sigma>1$, it holds that 
 \begin{equation}\label{4.64}\begin{aligned}&\left|\int_0^t\int_{\mathbb{R}^n_+}[L_{ij}(x,t)-L_{ij}(x,t-s)]G^{(1)}_{t-s}(y_n)u_nu_j\,dy\,ds\right|\\\leq&\int_0^t\int_{\mathbb{R}^n_+} \left(\int_0^1s|\partial_tL_{ij}(x,t-\eta s)|d\eta\right)(t-s)^{-\frac{1}{2}}e^{-\frac{y_n^2}{4(t-s)}}u_nu_j\,dy\,ds\\\leq &C\frac{1}{|x|^{n+2}}\int_0^t\int_{\mathbb{R}^n_+}s(t-s)^{-\frac{1}{2}}\lVert \boldsymbol{u}\rVert^2_{L^2(\mathbb{R}^n_+)}\,ds\\\leq &C\frac{t^{\frac{1}{2}}(1+\sqrt{t})^{-n}}{|x|^{n+2}}.\end{aligned}.\end{equation}

\textbf{Step 2.  Linear part $\boldsymbol{v}=e^{-t\mathbb{A}}\boldsymbol{a}$.}
It follows from (\ref{3.26})--(\ref{3.14}) that for any $t>0$ and $i=1,2,\dots,n$, one has\begin{equation}
    \begin{aligned}
        v_i(x,t)&=\int_{\mathbb{R}^n_+}[G_t(x-y)-G_t(x-y^*)]a_i(y)\,dy\\&\quad-\int_{\mathbb{R}^{n}_+}(\nabla_{x'}'\partial _{y_n}M^*_i(x',x_n,0',y_n,t)\cdot y' )a_n(y',y_n)\,dy'\,dy_n\\&\quad+\int_{\mathbb{R}^n_+}\partial_{y_n}\left(\int_0^1y'\nabla^2_{y'}M^*_i(x,\eta y',y_n,t)y'^T(1-\eta)d\eta\right) a_n(y',y_n)\,dy'\,dy_n\\&\quad-\sum_{k=1}^{n-1}4\partial_{x_k}\int_0^\infty\partial_{x_i}\left[G_{t+\tau}^{(n-1)}(x')G_\tau^{(1)}(x_n)\right]
\,d\tau\int_{\mathbb{R}^n_+}G_t^{(1)}(y_n)y_k a_n(y',y_n)\,dy'\,dy_n\\&=J_1+J_2+J_3+J_4.
    \end{aligned}
\end{equation}It follows from the pointwise estimates of heat kernel (\ref{heat}) that for $|x|>1$ and $t>0$, one has
\begin{equation}
    \begin{aligned}
        |J_1|&\leq\left|\int_{|y|\leq\frac{|x|}{2}}[G_t(x-y)-G_t(x-y^*)]a_i(y)\,dy\right|+\left|\int_{|y|>\frac{|x|}{2}}[G_t(x-y)-G_t(x-y^*)]a_i(y)\,dy\right|\\&\leq C\frac{t^{\frac{\theta-n}{2}}}{(|x|+\sqrt{t})^\theta}\int_{\mathbb{R}^n_+}|a_i(y)|\,dy +C\frac{1}{|x|^\theta}\int_{\mathbb{R}^n_+}|G_t(x-y)-G_t(x-y^*)|\,dy\\&\leq C\frac{1+t^{\frac{\theta-n}{2}}}{|x|^\theta},
    \end{aligned}
\end{equation}
where (\ref{cond:decay}) has been used.
Note that $|J_1|$ is uniformly bounded with respect to  $x$ and $t$, one has \begin{equation}|J_1|\leq C\frac{1+t^{\frac{\theta-n}{2}}}{(1+|x|)^\theta}.\end{equation}

Recall the pointwise estimates (\ref{2.9}) of $M(x,y,t),$ we obtain for $x\in\Lambda_{\sigma}$ with $\sigma>1$ and $t>0$
\begin{equation}\label{4.71}
    |J_2|\leq C\int_{\mathbb{R}^{n}_{+}}(x_n+\sqrt{t})^{-1}(|x'| + x_n + y_n + \sqrt t )^{-n}|y'||a_n(y)|\,dy\leq \frac{C}{|x|^{n+1}} 
\end{equation}\label{4.72}and\begin{equation}
    |J_3|\leq C\int_{\mathbb{R}^{n}_{+}}t^{-\frac{1}{2}}\left(\int_0^1(|x'-\eta y'| + x_n + y_n + \sqrt{t} )^{-(n+1)}(1-\eta)d\eta\right)|y'|^2|a_n(y)|\,dy\leq \frac{Ct^{-\frac{1}{2}}}{|x|^{n+1}}.
\end{equation} 
By the same argument as (\ref{arg}),  we find that if the coefficient of $L_{ii}(x,y,t)$
 $$\mathscr{A}_i[\boldsymbol{a},\boldsymbol{u}](t)\neq 0\quad \text{for some }i=1,2,\dots,n-1, $$
 there exist  $\widetilde{M}(t)>1$ large enough and $\widetilde{C}(t)>0$ such that for $x\in \Lambda_{\widetilde{M}(t)}\cap \left(B(0,\widetilde{M}(t))\right)^c$, one has
\begin{equation}
\frac{\widetilde{C}(t)}{|x|^n}\leq|J_4-I_{143}|=\left|\sum_{j=1}^{n-1}L_{ij}(x,t)\mathscr{A}_j[\boldsymbol{a},\boldsymbol{u}](t)\right|\leq \frac{\widetilde{C}(t)}{|x|^n}.
\end{equation}

 Hence, for fixed $t>0$, all the other terms decay faster than order $n$ for $x\in \Lambda_\sigma\cap \left(B(0,\sqrt{t})\right)^c$ with $\sigma>1$ except $J_4$ and $I_{143}$.  Thus, for any $i<n$, if 
 $\mathscr{A}_i[\boldsymbol{a},\boldsymbol{u}](t)\neq 0,$
  there exist $M(t)>1$ large enough and $C(t) > 0$ such that  $$|\boldsymbol{u}|\geq|u_i|>\frac{C(t)}{|x|^n} \quad\text{for }  x\in\Lambda_{M(t)}\cap \left(B(0,M(t))\right)^c.$$ This finishes the proof of the Case (2) of Theorem \ref{thm:1.2}.
      \end{proof}

\section{Short time behavior for solutions of the Navier--Stokes system in $L^1(\mathbb{R}^n_+)$}\label{sec:6}

 In this section, we first study the initial vanishing rate of leading terms
$\sum_{j=1}^{n-1}L_{ij}(x,t)\mathscr{A}_j(t)$
as $t$ tends to zero. Next we establish the second sufficient condition for $\boldsymbol{u}(\cdot,t)\notin L^1(\mathbb{R}^n_+)$ for $t$ near zero.

\begin{proof}[Proof of  Theorem~\ref{thm:1.3}] 
If $\boldsymbol{a}\in L^2_\sigma(\mathbb{R}^n_+)$ satisfies (\ref{p1}),  then one has for $t>0$ that
$$\begin{aligned}
    \left|\int_{\mathbb{R}^n_+} G^{(1)}_t(y_n)y_ja_n(y)\,dy\right|&\leq\int_0^\infty G^{(1)}_t(y_n)\int_{\mathbb{R}^{n-1}}|y_j a_n(y)|\,dy'\,dy_n\\&\leq C \int_0^\infty G^{(1)}_t(y_n)\frac{y_n^b}{(1+y_n)^b}\,dy_n\\&\leq Ct^{\frac{b}{2}}\int_0^\infty \lambda^be^{-\frac{\lambda^2}{4}}\,d\lambda\\&\leq Ct^{\frac{b}{2}},
\end{aligned}$$
    and\begin{equation}\label{6.29}\begin{aligned}
       &\left| \int_0^t\int_{\mathbb{R}^n_+}G_{t-s}^{(1)}(y_n)u_nu_j(y,s)\,dy\,ds\right|\leq \int_0^t(t-s)^{-\frac{1}{2}}\left|\int_{\mathbb{R}^n_+}u_nu_j(y,s)\,dy\right|\,ds \leq C\Vert \boldsymbol{u}\Vert^2_{L^2(\mathbb{R}^n_+)}t^{\frac{1}{2}}.
    \end{aligned}\end{equation}
    Then  there exists a small $T'>0$ such that  $$\left|\mathscr{A}_j[\boldsymbol{a},\boldsymbol{u}](t)\right|\leq Ct^{\min\{\frac{1}{2},\frac{b}{2}\}} \quad\text{for } t\in (0,T').$$

    If $\boldsymbol{a}\in L^2_\sigma(\mathbb{R}^n_+)$ satisfies (\ref{p2}),  we also have
    $$\left|\int_{\mathbb{R}^n_+} G^{(1)}_t(y_n)y_ja_n(y)\,dy\right|\leq Ct^{\frac{b_2}{2}}.$$ It follows from Lemma \ref{lem:2.6} and $\bar{b}>n$ that \begin{equation}\label{u0}|\boldsymbol{u}(x)|\leq \frac{Cx_n^{b_1}}{(1+x_n)^{b_1}(1+|x|)^{n}}\quad\text{for }t\in (0 , T_3).\end{equation}
    Therefore, one has\begin{equation}\label{6.30}\begin{aligned}
       &\left| \int_0^t\int_{\mathbb{R}^n_+}G_{t-s}^{(1)}(y_n)u_nu_j(y,s)\,dy\,ds\right|\\\leq &\int_0^t\int_0^\infty G^{(1)}_{t-s}(y_n)\left|\int_{\mathbb{R}^{n-1}}u_nu_j(y,s)\,dy'\right|\,dy_n\,ds\\\leq &C\int_0^t\int_0^\infty G^{(1)}_{t-s}(y_n)\frac{y_n^{2b_1}}{(1+y_n)^{2b_1}}\,dy_n \,ds\\\leq &C\int_0^t (t-s)^{b_1}\int_0^\infty \lambda^{2b_1}e^{-\frac{\lambda^2}{4}}\,d\lambda \,ds\\\leq &Ct^{b_1+1}.
    \end{aligned}\end{equation}
    Then  there exists a small $T''>0$ such that $$\left|\mathscr{A}_j[\boldsymbol{a},\boldsymbol{u}](t)\right|\leq Ct^{\min\{{b_1+1,\frac{b_2}{2}}\}}\quad \text{for } t\in (0,T'').$$ This finishes the proof of the theorem.
\end{proof}
Finally, we confirm the second sufficient condition such that $\boldsymbol{u}(\cdot,t)\notin L^1(\mathbb{R}^n_+)$. 
\begin{proof}[Proof of Case (3) in Theorem \ref{thm:1.2}] The estimate (\ref{1.10}) holds for $t\in (0,T_3)$, which follows directly from Lemma \ref{lem:5.1}.
   Hence  $\boldsymbol{u}(\cdot,t)\notin L^1(\mathbb{R}^n_+)$ provided  $\mathscr{A}_j[\boldsymbol{a},\boldsymbol{u}](t)\neq 0$ for some $j=1,2,\dots,n-1$.

    If $\boldsymbol{a}(x)$ satisfies (\ref{p0}),
  similar to (\ref{6.30}), using (\ref{u0}),  one has
    \begin{equation}\label{uu}\left| \int_0^t\int_{\mathbb{R}^n_+}G_{t-s}^{(1)}(y_n)u_nu_j(y,s)\,dy\,ds\right|\leq Ct^{b_1+1}.\end{equation}

    Note that $\boldsymbol{a}$ satisfies (\ref{condition1}) or (\ref{condition2})
   and for $j=1,2,\dots,n-1$ \begin{equation}\label{6.31}
        \begin{aligned}
            &\int_{\mathbb{R}^n_+} G_t^{(1)}(y_n)y_ja_n(y)\,dy\\=&\int_0^\infty G_t^{(1)}(y_n)\int_{\mathbb{R}^{n-1}}y_ja_n(y)\,dy'\,dy_n\\=&\int_0^{\epsilon_0} G_t^{(1)}(y_n)\int_{\mathbb{R}^{n-1}}y_ja_n(y)\,dy'\,dy_n+\int_{\epsilon_0}^\infty G_t^{(1)}(y_n)\int_{\mathbb{R}^{n-1}}y_ja_n(y)\,dy'\,dy_n.
        \end{aligned}
    \end{equation}
     For $t\in (0,\epsilon^2_0)$, the change of variables $\lambda=\frac{y_n}{\sqrt{t}}$ gives \begin{equation}\label{6.32}
        \begin{aligned}
            &\left|\int_0^{\epsilon_0} G_t^{(1)}(y_n)\int_{\mathbb{R}^{n-1}}y_ja_n(y)\,dy'\,dy_n\right|=\int_0^{\epsilon_0} G_t^{(1)}(y_n)\left|\int_{\mathbb{R}^{n-1}}y_ja_n(y)\,dy'\right|\,dy_n\\\geq &C\int_0^{\epsilon_0} G_t^{(1)}(y_n)y_n^{b_2}\,dy_n=Ct^{\frac{b_2}{2}}\int_0^{\frac{\epsilon_0}{\sqrt{t}}}\lambda^{b_2}e^{-\frac{\lambda^2}{4}}\,d\lambda\geq Ct^{\frac{b_2}{2}}\int_0^{1}\lambda^{b_2}e^{-\frac{\lambda^2}{4}}\,d\lambda\geq Ct^{\frac{b_2}{2}}
        \end{aligned}
    \end{equation}
    and 
    \begin{equation}\label{6.33}
        \begin{aligned}          &\left|\int_{\epsilon_0}^\infty G_t^{(1)}(y_n)\int_{\mathbb{R}^{n-1}}y_ja_n(y)\,dy'\,dy_n\right|\leq\left|\int_{\epsilon_0}^\infty G_t^{(1)}(y_n)\frac{y_n^{b_1}}{(1+y_n)^{b_1}}\,dy_n\right|\leq\int_{\frac{\epsilon_0}{\sqrt{t}}}^\infty\lambda^{b_1}e^{-\frac{\lambda^2}{4}}\,d\lambda\leq Ce^{-\frac{\epsilon_0^2}{8t}}.
        \end{aligned}
    \end{equation}
    It follows from (\ref{6.31})--(\ref{6.33}) that for any $t\in (0,\epsilon_0^2)$ small enough\begin{equation}\label{a}\left|\int_{\mathbb{R}^n_+} G^{(1)}_t(y_n)y_ja_n(y)\,dy\right|>Ct^{\frac{b_2}{2}}.\end{equation}
    This, together with (\ref{uu}) and $b_2<2b_1+2$ implies that there exists a small $T_4\in (0,\min\{1,\epsilon_0^2,T_3\})$ such that for $t\in(0,T_4)$, $$|\mathscr{A}_j[\boldsymbol{a},\boldsymbol{u}](t)|\geq\left|\int_{\mathbb{R}^n_+}G_t^{(1)}(y_n)y_j a_n(y)\,dy\right|-\left|\int_0^t\int_{\mathbb{R}^n_+}G_{t-s}^{(1)}(y_n)u_nu_j(y,s)\,dy\,ds\right|>Ct^{\frac{b_2}{2}}-Ct^{b_1+1}>0.$$ This finishes the proof of Case (3) in Theorem \ref{thm:1.2}.
\end{proof}




\begin{appendices}
\appendixnumbering 
\renewcommand{\thesection}{\textbf{Appendix \Alph{section}. Construction of the initial data}\hspace{1em}}

\section{} 

In this appendix, we construct a large class of initial data which satisfy the condition in Case (3) of Theorem \ref{thm:1.2}.

The space of continuous functions on $\overline{\mathbb{R}^n_+}$ is denoted by
\[
C^0(\overline{\mathbb{R}^n_+}) := \big\{ f: \overline{\mathbb{R}^n_+} \to \mathbb{R} \;\big|\; \text{$f$ is continuous up to the boundary } \{x_n = 0\} \big\}.
\]
\begin{Lemma}\label{A.2}
For any $b\geq 1$ and nonempty bounded open set $K=K_1\times K_2\times\cdots\times K_{n-1} \subset \mathbb{R}^{n-1}$, there exists a divergence-free vector field $\boldsymbol{a}=(a_1,a_2,\dots,a_n) \in C^{0}(\overline{\mathbb{R}^n_+})$ such that:
\begin{enumerate}
    \item $\mathrm{supp}\{\boldsymbol{a}\} \subseteq K \times [0,1]$.
    \item The normal component satisfies $a_n(x) = x_n^b\psi_n(x') + o(x_n^b)$ as $x_n \to 0^+$.
    \item The tangential components  satisfy  $a_i(x) =b x_n^{b-1}\psi_i(x') + o(x_n^{b-1})$ as $x_n \to 0^+$ for $i=1,2,\dots,n-1$.
    
\end{enumerate}
Here $\psi=(\psi_1,\psi_2,\dots,\psi_n) \in C_c^\infty(\mathbb{R}^{n-1})$ satisfies $\int_{\mathbb{R}^{n-1}} x_1\psi_n(x') \,dx'\neq 0$.
\end{Lemma}

\begin{proof}
First, define 
\[\psi_n(x') = \partial_{x_1}\eta_1(x_1)\eta_2(x_2)\cdots\eta_{n-1}(x_{n-1})\] where $\eta_i \in C_c^\infty(K_i)$ and $\int_{K_i}\eta_i(x_i)\,dx_i\neq 0$ for $i=1,2,\dots,n-1,$
     and choose $\phi \in C_c^\infty([0,1))$ with $\phi(0)=1$. It holds that 
     \[\int_{\mathbb{R}^{n-1}} x_1\psi_n(x') \,dx'=-\prod_{i=1}^{n-1}\int_{K_i}\eta_i(x_i)\,dx_i\neq 0.\]
Then, define the normal component as
\[
a_n(x) = x_n^b\psi_n(x')\phi(x_n).
\]
Finally, the tangential components are constructed by solving $\nabla \cdot \boldsymbol{a} = 0$:
\[
a_1(x) = -\left(\int_{-\infty}^{x_1}\psi_n(y_1,x_2,\dots ,x_{n-1}) \,dy_1\right)\partial_{x_n}(x_n^b\phi(x_n)),\quad a_i=0\quad \text{with}\ i\neq 1,n.
\]
Hence the proof of the lemma is finished.
\end{proof}
\end{appendices}

{\bf Acknowledgments.}
This work was partially supported by National Key R\&D Program of China 2024YFA1013302. The research of Wang is supported by NSFC grant 1237122. The research of  Xie is partially supported by  NSFC grants 12571238 and 12426203. 

\bibliographystyle{abbrv}

\end{document}